\documentclass[oneside,12pt,a4paper,english]{amsart}
\usepackage{pdflscape}
\usepackage{booktabs}
\usepackage{graphicx}
\usepackage{amsmath}
\usepackage{amssymb}
\usepackage{multirow}
\usepackage[table]{xcolor}
\usepackage[english]{babel}
\usepackage{amsmath}
\usepackage{amssymb}
\usepackage{amsthm}
\usepackage{amstext}
\usepackage{hyperref}
\usepackage{xcolor}
\usepackage{diagbox}
\usepackage{placeins}
\usepackage[most]{tcolorbox} 
\usepackage{geometry}

\definecolor{Blueish}{HTML}{1c4197}

\newtcolorbox{Chaptersetting}[1][]{
	colback=black!5,
	fonttitle=\bfseries,
	title={#1},
	coltitle=black,
	sharp corners,
	borderline west={2pt}{0pt}{Blueish},
	enhanced,
}

\DeclareMathOperator{\R}{\mathbb{R}}

\DeclareMathOperator{\Z}{\mathbb{Z}}

\DeclareMathOperator{\C}{\mathbb{C}}

\DeclareMathOperator{\SU}{SU}
\DeclareMathOperator{\SL}{SL}
\DeclareMathOperator{\GL}{GL}
\DeclareMathOperator{\SO}{SO}
\DeclareMathOperator{\U}{U}

\DeclareMathOperator{\Sp}{Sp}
\DeclareMathOperator{\Spin}{Spin}
\renewcommand{\k}{\ensuremath{\mathfrak{k}}}
\renewcommand{\t}{\ensuremath{\mathfrak{t}}}

\renewcommand{\u}{\ensuremath{\mathfrak{u}}}

\renewcommand{\sp}{\ensuremath{\mathfrak{sp}}}
\DeclareMathOperator{\so}{\mathfrak{so}}
\DeclareMathOperator{\su}{\mathfrak{su}}

\renewcommand{\l}{\ensuremath{\mathfrak{l}}}

\newcommand{\h}{\ensuremath{\mathfrak{h}}}

\newcommand{\myrestriction}{{\,\restriction\,}}

\DeclareMathOperator{\tr}{tr}
\DeclareMathOperator{\Ad}{Ad}

\DeclareMathOperator{\scal}{scal}
\DeclareMathOperator{\Ric}{Ric}

\DeclareMathOperator{\id}{Id}  

\DeclareMathOperator{\Hom}{Hom}

\DeclareMathOperator{\mult}{mult}

\renewcommand{\hat}[1]{\widehat{#1}}

\theoremstyle{plain}
\newtheorem{thm}{Theorem}[section]
\newtheorem*{thm*}{Theorem}
\newtheorem{lem}[thm]{Lemma}
\newtheorem*{lem*}{Lemma}
\newtheorem{stand}{Standing Assumption}
\newtheorem{prop}[thm]{Proposition}
\newtheorem{cor}[thm]{Corollary}

\theoremstyle{definition}
\newtheorem{rem}[thm]{Remark}
\newtheorem{dfn}[thm]{Definition}

\usepackage{graphicx}
\usepackage{tabularx}

\usepackage{tikz}

\begin{document}

	\title{Explicit Laplace Spectra of Homogeneous Principal Bundles}
    	\author{Ilka Agricola, Leandro Cagliero and Jonas Henkel}
	
	\address{\hspace{-0.8cm} 
\; Ilka Agricola, Jonas Henkel,\newline
Fachbereich Mathematik und Informatik,
Philipps-Universit\"at Marburg,\newline
Campus Lahnberge,
35032 Marburg, Germany\newline 
{\normalfont\ttfamily agricola@mathematik.uni-marburg.de, henkelj@mathematik.uni-marburg.de}
\vspace{0.2cm}\newline
Leandro Cagliero,\newline
CIEM-CONICET and FAMAF-Universidad Nacional de C\'ordoba,\newline
Medina Allende s/n, Ciudad Universitaria,
5000 C\'ordoba, Argentina;\newline 
Guangdong Technion Israel Institute of Technology, \newline 241
Daxue Road, Shantou, China\newline
{\normalfont\ttfamily leandro.cagliero@unc.edu.ar}}


	\maketitle
    \begin{center}
   \today
    \end{center}
    
 \begin{abstract}
We present a unified representation-theoretic method to compute the Laplace-Beltrami spectrum on homogeneous principal bundles. For this setting, we introduce a multi-parameter family of metric deformations called generalized canonical variations. Building upon the geometric realization of such fibrations as naturally reductive spaces, we establish a simplified spectral branching criterion. We apply this method to derive the full spectra—yielding all eigenvalues and multiplicities—for several prominent geometric families. Specifically, we compute the full spectra for the entire classical series of homogeneous 3-$(\alpha,\delta)$-Sasaki manifolds (Types A, B, C, and D) and for real and complex Stiefel manifolds over Grassmannians. These explicit formulas provide the analytical data to investigate related problems in geometric analysis. As an application, we classify the scalar stability of these spaces under Perelman's $\nu$-entropy and, for the 3-$(\alpha,\delta)$-Sasaki manifolds, determine the exact thresholds for Yamabe bifurcations.\end{abstract}

\section{Introduction}\label{sec:Intro}

\subsection*{Generalized Canonical Variations}
Understanding the Laplace–\allowbreak Beltrami spectrum is a key issue in geometric analysis because it is closely linked to the geometry, isometries and deformations of a Riemannian manifold. In this paper, we study the Laplace-Beltrami operator on \emph{homogeneous principal bundles} $P \to B$. Specifically, we focus on the geometric setting in which a compact Lie group $G$ already acts transitively on the total space $P$, such that the bundle structure arises from a $G$-equivariant fibration $\pi: G/K \to G/L$. 
The geometric fiber of the submersion is the homogeneous space $F \cong L/K$, which naturally carries the structure of a compact Lie group.

Our primary objects of study are families of $G$-invariant metrics obtained by rescaling the background metric along the fiber. If the structure group $F$ splits into commuting simple factors, our approach naturally accommodates independent scaling parameters for each factor, leading to what we refer to as \emph{generalized canonical variations}. From a geometric viewpoint, starting with a normal homogeneous metric ensures that these fibrations automatically have totally geodesic fibers, placing them firmly in the geometric context studied by Bérard-Bergery and Bourguignon \cite{Bergery_et_al}, (see also \cite{BessonBordoni90}).

\subsection*{Context and Motivation}
Explicit knowledge of the complete Laplace spectrum on such bundles has direct applications in theoretical physics and global variational problems.

In string theory and M-theory, compactifications on manifolds with special geometries (such as 3-Sasaki manifolds) are heavily studied via the AdS/CFT correspondence. The exact eigenvalues of the Laplace-Beltrami operator on the internal manifold determine the mass spectrum of the Kaluza-Klein modes (see, e.g., \cite[Eq. (6.18)-(6.20)]{Fabbri_et_al_2000}).

In geometric analysis, the family of homogeneous metrics $\{g_{\mathbf{t}}\}$ arising from a canonical variation naturally provides a 1-parameter family of trivial solutions to the Yamabe problem. Bettiol and Piccione \cite{BettiolPiccione13} demonstrated that on Berger spheres, symmetry-breaking bifurcations occur exactly when an eigenvalue of the total space Laplacian crosses a specific scalar curvature threshold. More recently, Otoba and Petean \cite{OtobaPetean20} established abstractly that an infinite sequence of such bifurcations arises along canonical variations, provided the initial fibration consists of spaces with constant scalar curvature. This bifurcation theory has subsequently been applied to various settings, including product manifolds \cite{deLimaPiccioneZedda12} and to maximal flag manifolds \cite{GramaLima20}. 
While the first eigenvalue $\eta_1$ determines the overall stability and local rigidity of these solutions (cf. \cite{1BLP22} for distance spheres and \cite{2BLP22} for homogeneous metrics on CROSSes), knowing the \emph{entire} spectrum is necessary to locate all higher degeneracy points and to completely map out the multiple branches of non-homogeneous Yamabe metrics.

The scalar spectrum also plays a crucial role in studying the stability of Einstein metrics. According to Cao and He \cite{CaoHe15}, the lowest eigenvalue $\eta_1$ governs the linear stability of Perelman's $\nu$-entropy against conformal deformations. Specifically, instability occurs if $\eta_1 < 2\Lambda$, where $\Lambda$ is the Einstein constant. In this case, we refer to the metric as \textit{strictly scalar-unstable}, and if $\eta_1 > 2\Lambda$, we refer to it as \textit{strictly scalar-stable}. While establishing stability against tensor perturbations is a highly non-trivial problem (recently advanced by Schwahn \cite{Schwahn22} for symmetric spaces), establishing full $\nu$-stability additionally requires determining the first scalar eigenvalue or finding appropriate estimates from below. This necessity has motivated recent research, ranging from the analysis of standard metrics on homogeneous spaces \cite{LauretLauret23, LauretTolcachier25} to the study of deformed families like canonical variations \cite{WangWang22}.

However, explicit computations of the entire spectrum of deformed metrics are still a rarity. In the literature, full eigenvalue computations for canonical variations have been largely confined to spheres and projective spaces. Early seminal works focused on determining the first eigenvalue $\eta_1$ for Hopf fibrations \cite{Urakawa79, Tanno79, Tanno80}. This line of research remains highly active, culminating in recent exact descriptions of the lowest eigenvalues for homogeneous 3-spheres \cite{Lauret19}, the first eigenvalue of general homogeneous CROSSes \cite{2BLP22}, and the full spectra of distance spheres \cite{1BLP22}. However, for bundles involving more intricate structure groups, or for non-spherical total spaces such as Stiefel manifolds, explicit descriptions of the \emph{entire} spectrum are very rare. Consequently, stability analyses on these spaces often rely on bounds or remain open problems.

\subsection*{Methodological Advance} To derive explicit formulas for the complete spectra of these families, we build upon the geometric framework established in \cite{AgricolaHenkel25}.
This framework realizes generalized canonical variations on $G/K$ algebraically as naturally reductive metrics on the product space $(G \times H)/K'$. While it provides a useful theoretical correspondence, its practical application is combinatorially challenging when identifying $K'$-spherical representations. In \cite{AgricolaHenkel25}, this task relied on ad-hoc counting arguments or explicit numerical evaluations of Kostant's partition function to find invariant vectors. This is a process that, while effective for isolated examples, does not readily yield general formulas for entire families of manifolds.
 To overcome this, we demonstrate that the search for $K'$-spherical vectors can be systematically structured as a sequential algebraic branching problem. We establish a branching criterion that exploits the product structure of the base isotropy group $L = K \cdot H$. By first restricting the $G$-representation to $L$, the condition for diagonal invariance simplifies via Schur's lemma: the problem reduces to identifying the dual fiber representation within the $K$-trivial subspace of this restriction.
By combining specific representation-theoretic tools—most notably the reciprocity principle (Lemma \ref{lem:reciprocity}) and the Littlewood-Richardson rule—we derive explicit algebraic formulas for the eigenvalues and their multiplicities.

\subsection*{Main Results}
We apply this method to the following geometric classes:

\begin{itemize}
    \item \textit{Classical homogeneous 3-$(\alpha, \delta)$-Sasaki manifolds:} We compute the complete spectra for all classical series (Types A, B, C, and D). This generalizes the prior computation for the Aloff-Wallach manifold $W^{1,1}$ \cite{AgricolaHenkel25}. While the spectrum for the Type C series was previously known \cite[Thm. 3.8]{2BLP22}—which our approach naturally recovers (cf. Remark \ref{rem: comparison to bettiol})—the spectra for Types A, B, and D are newly established.
    
    \item \textit{Stiefel manifolds:} We determine the spectra for both real and complex Stiefel manifolds, treated as principal bundles over their respective Grassmannians. 
\end{itemize}
In Section \ref{sec:applications}, we use these spectral results to analyze the stability of the associated Einstein metrics. For Perelman's $\nu$-entropy, we explicitly evaluate the scalar stability condition, proving that both Einstein metrics on Type A 3-$(\alpha,\delta)$-Sasaki manifolds are strictly scalar-stable. The squashed Einstein metrics on Type B/D and Type C 3-$(\alpha,\delta)$-Sasaki manifolds, as well as the investigated Einstein metrics on Stiefel manifolds, are strictly scalar-unstable (and therefore dynamically unstable under the Ricci flow). Finally, we determine the threshold equations governing symmetry-breaking Yamabe bifurcations along these families.
\section*{Acknowledgments}
This project was initiated during the joint workshop \emph{`Lie Theory in Geometry, Algebra and Analysis'} held in Córdoba, Argentina, in March 2025. The first and third authors gratefully acknowledge the hospitality and support of the National University of Córdoba during this event. Their stay was funded by the Deutsche Forschungsgemeinschaft (DFG, German Research Foundation) - 541920696. Furthermore, the second author wishes to express his gratitude to the University of Marburg for the support to this project and its very warm hospitality during his visit in September 2025.
The second author is partially supported by grants SeCyT-UNC 33620230100117CB and PIP 11220210100597CO.

Finally, the authors thank Emilio Lauret for pointing out the references \cite{Chami_Spn}, \cite{Chami}, and \cite{BenHalima07}.

\subsection*{Declaration of Generative AI in the writing process}
During the preparation of this work, the third-named author used generative AI tools (Gemini 2.5/3 Pro, ChatGPT 4o-5.5 Thinking, Grok 4, DeepL) for language refinement, structural editing, term evaluation/term comparisons, literature search and analysis. After using these tools, the authors rigorously reviewed, verified, and edited all content. The authors take full responsibility for the validity and final presentation of the publication.

\section{Spectral Theory on Principal Bundles}
\label{sec:spectral_theory}
To compute the exact Laplace-Beltrami spectrum on homogeneous principal bundles, we first need to establish the precise geometric setup that allows us to deform the metric along the fiber.

\subsection{Canonical Variations and the Principal Bundle Structure}

The explicit computation of the Laplace-Beltrami spectrum on homogeneous spaces $G/K$ is classically well-understood when the metric is standard normal homogeneous, meaning it is induced by a bi-invariant metric $Q$ on $G$. However, many geometrically significant families of metrics—such as those on 3-$(\alpha,\delta)$-Sasaki manifolds or Stiefel manifolds—arise as specific deformations of these standard metrics. We first establish the algebraic framework for these deformations.

Let $G$ be a compact, connected Lie group equipped with a bi-invariant metric $Q$, and let $K \subset L \subset G$ be connected, closed subgroups. We consider the orthogonal complement of the Lie algebra $\k$ within $\mathfrak{l}$ with respect to $Q$, denoted by $\h := \k^\perp \cap \mathfrak{l}$. The following lemma relates the commutativity of $\h$ and $\k$ to the normal subgroup structure of $K$.

\begin{lem}\label{lem:algebra_equivalence}
  Under the setup described above, the following conditions are equivalent:
    \begin{enumerate}
        \item[(i)] The subspace $\h$ commutes with $\k$, i.e., $[\k, \h] = \{0\}$.
        \item[(ii)] $K$ is a normal subgroup of $L$ $(K \triangleleft L)$.
    \end{enumerate}
    If these equivalent conditions hold, $\h$ is a Lie subalgebra of $\l$, and the subgroup $L$ factorizes as $L = K \cdot H$ with $[K, H] = \{e\}$, where $H$ is the connected subgroup generated by $\h$.
\end{lem}
\begin{proof}
    (i) $\Rightarrow$ (ii):
    For an arbitrary $X\in \l$, we have the decomposition $X=X_{\k}+X_{\h}\in \k\oplus \h$. Using that $[\k, \h] = \{0\}$, we have $[X,\k]=[X_{\k}+X_{\h},\k]=[X_{\k},\k]\subset \k$. Hence, $\k$ is an ideal in $\mathfrak{l}$, which implies $K \triangleleft L$ as Lie groups.\\
    (ii) $\Rightarrow$ (i): If $K \triangleleft L$, then $\k$ is an ideal in $\mathfrak{l}$, i.e. $[\k, \mathfrak{l}] \subset \k$. Let $X \in \k, Y \in \h$, and $Z \in \k$. As the metric $Q$ is bi-invariant, it satisfies $Q([X, Y], Z) = -Q(Y, [X, Z])$. Since $[X, Z] \in \k$ and $Y \in \h = \k^\perp \cap \mathfrak{l}$, the right-hand side is zero. Therefore, $[X, Y] $ lies in the orthogonal complement of $\k$. As $\k$ is an ideal, $[X,Y]$ also lies in $ \k$ , hence $[X,Y]=0$. 

     Furthermore, if $[\k, \h] = \{0\}$, it follows for any $Y_1, Y_2 \in \h$ and $X \in \k$ that the bi-invariance of $Q$ yields $Q([Y_1, Y_2], X) = Q(Y_1, [Y_2, X]) = 0$. Thus, $\h$ is closed under the commutator, making it a Lie subalgebra. The condition $[\k, \h] = \{0\}$ integrates directly to $[K, H] = \{e\}$ for the connected subgroups, yielding the factorization $L = K \cdot H$. \\
\end{proof}

While this establishes the algebraic mechanism to deform a metric along the subalgebra $\h$, this specific setup has an important geometric consequence: it canonically equips the homogeneous space $G/K$ with the structure of a principal bundle.

\begin{lem}\label{lem:bundle_equivalence}
  Let $G$ be a compact Lie group and $G/L$ be the base of a homogeneous principal $F$-bundle $P \to G/L$. The following conditions are equivalent:
    \begin{enumerate}
        \item[(i)] The natural action of $G$ extends to act transitively on the total space $P$.
        \item[(ii)] There exists a closed normal subgroup $K \triangleleft L$ such that the total space $P$ is $G$-equivariantly diffeomorphic to the homogeneous space $G/K$. Under this identification, the bundle projection becomes the natural submersion $\pi: G/K \to G/L$, and the geometric fiber is isomorphic to the Lie group $L/K \cong F$.
    \end{enumerate}
\end{lem}
\begin{proof}
    (ii) $\Rightarrow$ (i): If $P \cong G/K$, the group $G$ clearly acts transitively on the total space by left multiplication. \\
    (i) $\Rightarrow$ (ii): Any homogeneous principal $F$-bundle over $G/L$ can be represented as the associated bundle $P = G \times_\phi F$, defined by a structure homomorphism $\phi: L \to F$ (cf. \cite[Sec. 4]{Wang58}). Take an arbitrary element $b\in F$ together with the neutral element $e\in G$ and consider the respective equivalent class $[e,b]\in P$. Its $G$-orbit is given by the set $\{[a,b]\mid a\in G\}$. As $G$ acts transitively on $P$, there exists an $a\in G$ such that $[a,b]=[e,e_F]$, where $e_F$ is the neutral element in $F.$  By definition of the equivalence relation, this requires $a^{-1} \in L$ and $\phi(a) = b^{-1}$, i.e. $\phi(a^{-1}) = b$. Thus, $\phi$ is surjective. 
    
    The stabilizer of $[e, e_F]$ under the $G$-action consists of all elements $a \in L$ satisfying $\phi(a^{-1})e_F = e_F$, which forces $a \in \ker(\phi)$. Defining $K := \ker(\phi)$ ensures that $K$ is a normal subgroup of $L$ $(K \triangleleft L)$ and identifies the total space as $P \cong G/K$. Since $\phi:L\rightarrow F$ is surjective, the first isomorphism theorem implies that the fiber $F$ is isomorphic to the Lie group $L/K$.
\end{proof}

Combined, Lemmas \ref{lem:algebra_equivalence} and \ref{lem:bundle_equivalence} show that the algebraic condition of commuting subalgebras is equivalent to the geometric structure of a homogeneous principal bundle. This equivalence directly motivates our main setup.

\begin{stand}\label{stand:main_setup}
    Let $G$ be a compact, connected Lie group equipped with a bi-invariant metric $Q$. We assume a $G$-equivariant fibration $\pi: G/K \to G/L$ defined by the inclusion of connected subgroups $K \triangleleft L \subset G$. 
    
    Consequently, the base isotropy factorizes as $L = K \cdot H$, where $H$ is the connected subgroup generated by $\h = \k^\perp \cap \mathfrak{l}$. We further assume that the fiber group $H$ decomposes into a product of mutually commuting, connected subgroups $H = H_1 \times \dots \times H_s$, whose Lie algebras $\h_i$ are mutually orthogonal with respect to $Q$. Let $\mathfrak{m}$ denote the $Q$-orthogonal complement of $\mathfrak{l}$ in $\mathfrak{g}$.
\end{stand}

Based on this standing assumption, we define a multi-parameter family of $G$-invariant metrics on the total space $G/K$.

\begin{dfn}
    The \emph{generalized canonical variation} $g_{\mathbf{t}}$ on the principal bundle $G/K$ is defined by the orthogonal direct sum:
    \begin{equation}
        g_{\mathbf{t}} = t_0 Q|_{\mathfrak{m}} \,\, \oplus \,\, t_1 Q|_{\h_1} \,\, \oplus \,\, \dots \,\, \oplus \,\, t_s Q|_{\h_s},
    \end{equation}
    where $t_0 > 0$ uniformly scales the horizontal space and $t_1, \dots, t_s > 0$ independently scale the mutually orthogonal vertical distributions. As done several times in this article, we abbreviate tuples such as $(t_1,...,t_s)$ using the vector notation $\textbf{t}$.
\end{dfn}

To compute the exact Laplace-Beltrami spectrum for this entire family of non-normal metrics, we exploit the commutativity $[K,H]=\{e\}$ (guaranteed by Standing Assumption \ref{stand:main_setup}) to lift the problem to a larger product group, where the metric becomes naturally reductive.

\begin{thm}[cf. {\cite[Thm. 2.23]{AgricolaHenkel25}}]\label{thm:naturally_reductive_realization}
    Equipped with the generalized canonical variation $g_{\mathbf{t}}$, the total space of the homogeneous principal bundle $(G/K, \, g_{\mathbf{t}})$ is equivariantly isometric to a naturally reductive homogeneous space presented as a quotient of the product group:
    \[ \frac{G \times H_1 \times \dots \times H_s}{K'}, \]
    where the isotropy group $K' = \{(k h_1 \dots h_s, \, h_1, \dots, h_s) \mid k \in K, h_i \in H_i\}$ is diagonally embedded. 
\end{thm}

This realization unifies the symmetries of the deformed metric and allows the Laplace-Beltrami operator to be expressed entirely via the generalized Casimir operators of the product group, leading to Theorem \ref{thm:operational_branching}
\subsection{The branching condition}
We now express the sphericality condition of the naturally reductive realization in terms of branching rules. By restricting representations of $G$ to the base isotropy group $L = K \cdot (H_1 \dots H_s)$, the computation of the spectrum reduces to evaluating specific branching multiplicities. 

We first fix our representation-theoretic conventions. Throughout this paper, for any Lie group $G$, the notation $\widehat{G}$ denotes the set of equivalence classes of finite-dimensional, irreducible complex representations of $G$. Whenever $G$ is compact, these representations are naturally assumed to be unitary, equipped with a $G$-invariant inner product. 
\begin{dfn}
    Given a closed subgroup $K \subset G$, a representation $\varrho \in \widehat{G}$ is called \emph{$K$-spherical} if the space of $K$-fixed vectors is non-trivial, i.e., $\dim(V_\varrho^K) > 0$.
\end{dfn}

Since $K$ and $H$ commute, the irreducible representations of $L$ can be identified with tensor products of irreducible representations of $K$ and the factors $H_i$, provided they act trivially on the discrete intersection kernel $K \cap H$.

\begin{thm}\label{thm:operational_branching}
    Let $(G/K, g_{\mathbf{t}})$ be the generalized canonical variation of a homogeneous principal bundle as defined above. A tuple of irreducible representations 
    \[ (\varrho_G, \varrho_1, \dots, \varrho_s) \in \widehat{G} \times \widehat{H}_1 \times \dots \times \widehat{H}_s \]
    contributes to the Laplace-Beltrami spectrum (i.e. is $K'$-spherical) if and only if the dual representations $\varrho_i^*$ appear in the $K$-isotypic component of $\varrho_G$ restricted to the subgroup $L$. That is, the branching multiplicity must satisfy:
    \begin{align}\label{eq:branching_condition}
        \mult\left(1_K \otimes \varrho_1^* \otimes \dots \otimes \varrho_s^*, \quad \varrho_G\myrestriction_{L}\right) > 0,
    \end{align}
    where $1_K$ denotes the trivial representation of $K$.
    
    If this condition is met, the associated exact eigenvalue is given by the formula:
    \begin{equation}\label{eq:eigenvalue_formula}
        \eta = \frac{c_Q(\varrho_G) - \sum_{i=1}^s c_{Q\restriction \h_i}(\varrho_i)}{t_0} + \sum_{i=1}^s \frac{ c_{Q\restriction \h_i}(\varrho_i)}{t_i},
    \end{equation}
    where $c_Q(\varrho)$ denotes the eigenvalue of the Casimir operator associated with the bi-invariant metric $Q$, acting on the representation $\varrho$.
    
    The multiplicity of the eigenvalue $\eta$ corresponding to this tuple is given by:
    \[
        \mult(\eta) = \mult\left(1_K \otimes \varrho_1^* \otimes \dots \otimes \varrho_s^*, \quad \varrho_G\myrestriction_{L}\right)\cdot \dim(\varrho_G)\cdot  \prod_{i=1}^s \dim(\varrho_i).
    \]
\end{thm}
\begin{proof}
    By \cite[Thm. 2.23]{AgricolaHenkel25}, the spectrum of the naturally reductive space $(G \times H)/K'$ is entirely determined by the $K'$-spherical representations of the product group $G \times H$, meaning we must identify those irreducible representations $\varrho_G \otimes \varrho_H$ that contain a non-zero vector fixed under the action of $K' = \{(kh, h) \mid k \in K, h \in H\}$. If such a representation is found, the eigenvalue formula \eqref{eq:eigenvalue_formula} immediately follows from evaluating the generalized Casimir operator. Let us denote the factors of $H$ collectively by $\varrho_H = \varrho_1 \otimes \dots \otimes \varrho_s$.
    
    To find the $K'$-fixed vectors in $V_{\varrho_G} \otimes V_{\varrho_H}$, we proceed in two steps. First, we restrict $\varrho_G$ to the subgroup $L = K \cdot H$. Since $K$ and $H$ commute, this restriction decomposes into irreducible representations of the form $\mu \otimes \nu$, where $\mu \in \widehat{K}$ and $\nu \in \widehat{H}$. For a vector to be invariant under elements of the form $(k, e) \in K'$ (i.e., invariant under $K$ acting solely on the first factor), it must belong to a subspace where $\mu = 1_K$. Let $V_{\varrho_G}^K$ be this $K$-fixed subspace, which carries a residual representation of $H$, say $\sigma = \bigoplus_j m_j \nu_j$.
    
    Second, we require the vector to be invariant under the diagonal action of elements $(h, h) \in K'$. This translates to finding an $H$-invariant vector in the tensor product of the residual representation $\sigma$ and the fiber representation $\varrho_H$, i.e., we need the trivial representation $1_H$ to appear in $\sigma \otimes \varrho_H$.
    
    It is a standard fact from representation theory that for two irreducible representations $\nu$ and $\varrho_H$ of a compact Lie group, the trivial representation appears in their tensor product if and only if they are dual to each other: Explicitly, we identify $V_{\nu} \otimes V_{\varrho_H}$ with the space of linear maps $\Hom(V_{\nu}^*, V_{\varrho_H})$, endowed with the action $\Hom(V_{\nu}^*, V_{\varrho_H})\ni L \mapsto \quad\varrho_H(h) (L(\nu^{-1}(h)\:\cdot\:)$ of $h\in H$. The subspace of $H$-fixed vectors $(V_{\nu} \otimes V_{\varrho_H})^H$ is canonically isomorphic to the space of intertwining operators $\Hom_H(V_{\nu}^*, V_{\varrho_H})$. By Schur's Lemma, this space is non-trivial (specifically, one-dimensional) if and only if $\nu \cong \varrho_H^*$. 
    
    Therefore, the representation $\varrho_G \otimes \varrho_H$ is $K'$-spherical if and only if the dual representation $\varrho_H^*$ appears in the $K$-fixed part of $\varrho_G\myrestriction_{L}$, which is precisely condition \eqref{eq:branching_condition}.
\end{proof}
\begin{rem}
    We highlight three important aspects of Theorem \ref{thm:operational_branching}:
    \begin{enumerate}
        \item 
        The algebraic condition \eqref{eq:branching_condition} reflects the topology of the geometric fiber. By requiring that the tensor product interacts with the trivial representation $1_K$ over the subgroup $L = K \cdot H$, the condition automatically enforces that the product $\varrho_1 \otimes \dots \otimes \varrho_s$ acts trivially on the intersection $K \cap H$. This implicitly guarantees that the chosen fiber representation descends properly to the actual geometric fiber $F \cong H/(H \cap K)$.
        
        \item 
        The simplicity of the condition in Theorem \ref{thm:operational_branching} relies on Schur's lemma. Since we are computing the spectrum of the scalar Laplacian (acting on functions), we are searching for the \emph{trivial} representation $1_H$ in the diagonal tensor product $\sigma \otimes \varrho_H$. This immediately implies $\sigma \cong \varrho_H^*$. However, the underlying sequential branching mechanism is much broader. If one wishes to compute the spectrum of differential operators acting on sections of associated vector bundles (where the fiber group $H$ acts via a non-trivial representation $\tau$), one must instead find $\tau$ within the tensor product $\sigma \otimes \varrho_H$. While this no longer reduces to a duality condition, it remains a purely representation-theoretic task that can be systematically solved using Clebsch-Gordan coefficients (which is particularly straightforward to implement when the fiber is a low-rank group like $\SU(2)$ or $\U(2)$) or the Littlewood-Richardson rule.
        
        \item 
        In the specific case where the structure group $H$ is simple (i.e., a single scaling parameter $t_1)$, our spectral formula recovers the structure of the expression derived by Bettiol et al.\ \cite[Cor. 2.5]{1BLP22}, which has been used in the context of distance spheres in symmetric spaces.
    \end{enumerate}
\end{rem}

\section{Representation theoretic tools and reciprocity}\label{sec:the_tools}
To apply Theorem \ref{thm:operational_branching} to our target manifolds, we require explicit control over the restriction of representations for the classical Lie groups $\SU(n)$, $\SO(n)$, and $\Sp(n)$. The computation of the branching multiplicity $\mult(1_K \otimes \varrho_H^*, \varrho_G\myrestriction_{L})$ can be addressed using classical branching theorems.

However, evaluating these restrictions directly is often combinatorially demanding. In certain cases—most notably for fibrations involving unitary groups—we employ a reciprocity principle. This duality reduces the restriction problem to a tensor product decomposition, allowing us to utilize combinatorial tools such as the Littlewood-Richardson rule.

\begin{prop}[Fundamental Branching Rules, cf. {\cite[Ch. 8.1.1]{Goodman Wallach}}]\label{prop: branching GoWa}
	The irreducible representations of the classical groups satisfy the following fundamental branching rules:
	\begin{itemize}
		\item[\textbf{a)}] For $\GL(n)$: An irreducible, finite dimensional representation $\varrho_{\GL(n)}(z_{1},..,z_{n})$ decomposes as follows into irreducible representations of
		\begin{itemize}
			\item the subgroup $\GL(n-1) \times \GL(1)$ (embedded as block diagonal matrices):
			\begin{align*}
			\bigoplus_{z_{1}\geq x_{1}\geq z_{2}\geq \dots \geq x_{n-1}\geq z_{n}} \varrho_{\GL(n-1)}(x_{1},...,x_{n-1})\otimes \varrho_{\GL(1)}\left(\sum_{i=1}^{n}z_{i}-\sum_{i=1}^{n-1}x_{i}\right),
			\end{align*}
			\item the subgroup $\SL(n)\times\GL(1)$:
			$$ \varrho_{\SL(n)}(z_{1}-z_{n},...,z_{n-1}-z_{n})\otimes \varrho_{\GL(1)}\left(\sum_{i=1}^{n}z_{i}\right).$$
		\end{itemize}
		
        \item[\textbf{b)}] For $\Spin(2n+1) \downarrow \Spin(2n)$: The branching is multiplicity-free. An irreducible representation $\varrho_{\Spin(2n)}(x_1, \dots, x_n)$ occurs in the decomposition of $\varrho_{\Spin(2n+1)}(z_1, \dots, z_n)$ if and only if their highest weights satisfy the interlacing condition
		\begin{align*}
			z_1 \ge x_1 \ge z_2 \ge \dots \ge z_{n-1} \ge x_{n-1} \ge z_n \ge |x_n|.
		\end{align*}
		
		\item[\textbf{c)}] For $\Spin(2n) \downarrow \Spin(2n-1)$: The branching is multiplicity-free. An irreducible representation $\varrho_{\Spin(2n-1)}(x_1, \dots, x_{n-1})$ occurs in the decomposition of\linebreak $\varrho_{\Spin(2n)}(z_1, \dots, z_n)$ if and only if their highest weights satisfy the interlacing condition
		\begin{align*}
			z_1 \ge x_1 \ge z_2 \ge \dots \ge z_{n-1} \ge x_{n-1} \ge |z_n|.
		\end{align*}
	\end{itemize}
	For the spin groups, the representations must also be of the same type (i.e., both integral or both half-integral).
\end{prop}
For complex manifolds such as the Type A 3-$(\alpha,\delta)$-Sasaki manifolds and the complex Stiefel manifolds, the structure group $L$ embeds block-diagonally. A direct analysis of this restriction is complicated by the special unitary condition, which couples the determinants of the blocks. To overcome this, we lift the representations from $\SU(N)$ to $\U(N)$ and utilize a reciprocity relation between restriction multiplicities and tensor product multiplicities.\\
In this article, we identify a partition $\lambda = (\lambda_1, \dots, \lambda_k)$ with a dominant integral weight. For any general linear group $\GL(N)$ with rank $N \ge k$, we denote by $\varrho_{\GL(N)}(\lambda)$ the irreducible representation whose highest weight is given by appending $N-k$ zeros to the partition:
\[
\lambda \equiv (\lambda_1, \dots, \lambda_k, \underbrace{0, \dots, 0}_{N-k}).
\]
This convention allows us to smoothly relate representations of groups of different ranks, as required by the following reciprocity relation.

\begin{lem}[Reciprocity of multiplicities]\label{lem:reciprocity}
    Let $\mu, \nu, \lambda$ be partitions with depths at most $m$, $n$, and $m+n$, respectively. The following identity relates the branching multiplicity to the tensor product multiplicity:
    \begin{align*}
    &\mathrm{mult}
    \left(
        \varrho_{\GL(m)}(\mu) \otimes \varrho_{\GL(n)}(\nu), \quad
        \varrho_{\GL(m+n)}(\lambda)\myrestriction_{\GL(m)\times\GL(n)}
    \right)\\
    &=
    \mathrm{mult}
    \big(
        \varrho_{\GL(m+n)}(\lambda), \quad
        \varrho_{\GL(m+n)}(\mu) \otimes \varrho_{\GL(m+n)}(\nu)
    \big).
    \end{align*}
\end{lem}
\begin{proof}
This duality is a consequence of the properties of Littlewood-Richardson coefficients and can be found, for instance, in \cite[Thm. 9.2.3]{Goodman Wallach}. It implies that the branching rule for $\GL(m+n) \downarrow \GL(m) \times \GL(n)$ is governed by the same combinatorial coefficients as the tensor product of representations of $\GL(m+n)$.
\end{proof}

The reciprocity lemma reduces our spectral conditions to the decomposition of tensor products of $\GL(N)$ representations. The following theorem explicitly resolves the specific tensor product structure that will repeatedly arise in our computations.

\begin{prop}[A tensor product decomposition]\label{prop: tensor product decomposition}
    Assume $n\ge 3$, and let $a\ge b$ with $a+b=2k$ for non-negative integers $a,b,k$.
    Let $\mu$ and $\nu$ be the partitions
    \[
    \mu=(a, b),\qquad \nu=(k^{n-1}).
    \]
    Then the tensor product of representations of $\GL(n+1)$,
    \[
    \varrho_{\GL(n+1)}(\mu) \otimes \varrho_{\GL(n+1)}(\nu),
    \]
    is multiplicity free. The irreducible representation $\varrho_{\GL(n+1)}(\lambda)\in \hat{\GL(n+1)}$ appears in the decomposition if and only if
    \[
    \lambda=(k+i_1,k+i_2,k^{n-3},i_3,i_4)
    \]
    with $i_1+i_2+i_3+i_4=2k$ and
    \begin{align*}
    a-k & \le i_1 \le a, \\
    a-i_1 & \le i_3 \le k,\\
    0 & \le i_2 \le a-i_3, \\
    0 & \le i_4 \le a-i_1.
    \end{align*}
\end{prop}

\begin{proof}
From the Littlewood-Richardson rule we know that the multiplicity
\[
\mathrm{mult}
\big(
\varrho_{\GL(n+1)}(\lambda), 
\varrho_{\GL(n+1)}(\mu) \otimes \varrho_{\GL(n+1)}(\nu)
\big)
\]
is the number of ways the 
Young diagram for $\lambda$ can be obtained from the 
Young diagram for $\nu$ by a strict $\mu$-expansion. 
Recall that an $(a,b)$-expansion of a Young diagram is obtained by:
\begin{enumerate}
\item First adding $a$ boxes, according to Pieri's rule (not two boxes in the same column), and putting the integer 1 in each of these boxes.
\item Next, adding $b$, according to Pieri's rule, boxes with the integer 2.
\end{enumerate} 
This expansion is said to be strict if, when the integers 1 and 2 in the boxes are listed from right to left, starting with the top row and going down, we never encounter more 2s than 1s 
(this is the ``Yamanouchi word'' condition, see \cite[p. 456]{Fulton_Harris}).

For step (1) we must choose $i_1$ boxes, with $0\le i_1\le a$, to put in the first row, and the other $a - i_1$ boxes will go to the $n$-th row. 
To obtain a Young diagram it is required $a - i_1\le k$, that means
\[
a-k\le i_1\le a.
\]

To perform step (2) we must choose $i_2, j, i_4\ge 0$ boxes,
 with $i_2+j+i_4=b$, to add in rows 2, $n$ and $n+1$ respectively (by the strict condition,
 we cannot put boxes labeled with 2 in the 1st row).
The restrictions we have for $i_2, j, i_4\ge 0$ are:
\begin{enumerate}
\item[(i)] $i_2+j\le i_1$ by the strict condition (it follows $i_2\le i_1)$.
\item[(ii)] $a-i_1+j\le k$ to obtain a Young diagram.
\item[(iii)] $i_4\le a-i_1$ to avoid placing two boxes in the same column 
(it follows $i_4\le a-i_1+j)$.
\end{enumerate}
If we call $i_3=a-i_1+j$ (the number of boxes in row $n)$, 
we end up with 
\[
\lambda=(k+i_1,k+i_2,k^{n-3},i_3,i_4),
\] 
condition $j\ge 0$ is $a-i_1\le i_3$, 
condition $i_2+j+i_4=b$ is $i_1+i_2+i_3+i_4=a+b=2k$, 
and the previous three conditions become
\begin{enumerate}
\item[(i)] $i_2\le a-i_3$.
\item[(ii)] $i_3\le k$.
\item[(iii)] $i_4\le a-i_1$.
\end{enumerate}
This completes the proof.
\end{proof}
\section{The Spectrum of Positive Homogeneous 3-$(\alpha,\delta)$-Sasaki Manifolds}

The class of 3-$(\alpha,\delta)$-Sasaki manifolds, introduced by Agricola and Dileo \cite{AgricolaDileo20}, is a family of almost 3-contact metric manifolds generalizing classical 3-Sasaki geometry. We recall the central definition:

\begin{dfn}[\cite{AgricolaDileo20}]
    An \emph{almost 3-contact metric manifold} $(M, \varphi_i, \xi_i, \eta_i, g)$ of dimension $4d+3$ is called a \emph{3-$(\alpha,\delta)$-Sasaki manifold} if it satisfies
    \[
        d\eta_i = 2\alpha\Phi_i + 2(\alpha-\delta)\eta_j \wedge \eta_k
    \]
    for all triples $(i,j,k)$ obtained by cyclically permuting $(1,2,3)$, where $\alpha \neq 0$ and $\delta$ are real constants, and $\Phi_i(X,Y) = g(X, \varphi_i Y)$ are the fundamental 2-forms.
\end{dfn}

These manifolds are characterized by the fact that they admit a canonical metric connection with totally skew torsion. A comprehensive description of their geometric properties can be found in \cite{AgricolaDileoStecker21,AgricolaDileo20}. 

As detailed in \cite{AgricolaDileoStecker21}, every positive $(\alpha\delta > 0)$ homogeneous 3-$(\alpha,\delta)$-Sasaki manifold arises as the total space of a Riemannian submersion over a quaternionic Kähler manifold (a Wolf space). See also \cite{GoertschesRoschigStecker23} for a recent algebraic revisiting of the underlying 3-Sasaki classification. The vertical distribution is spanned by the three Reeb vector fields $\xi_i$. Consistent with our framework in Section \ref{sec:spectral_theory}, the geometric fiber $F$ is a space form of constant positive curvature that carries the structure of a compact Lie group. Specifically, $F \cong H / (H \cap K)$, where the structure group is $H = \SU(2)$ and the fiber is locally isometric to either $\SU(2)$ or $\SO(3)$. We note that simply connected, positive 3-$(\alpha,\delta)$-Sasaki manifolds are necessarily Spin manifolds.

By applying the general variation framework, these spaces $M = G/K$ can be realized algebraically as naturally reductive spaces of the form $(G \times H)/K'$, where $H = \SU(2)$ \cite{AgricolaHenkel25}.

The parameters $(\alpha, \delta)$ control the curvature properties of the manifold. The metric $g$ is Einstein for exactly two specific choices of parameters:

\begin{prop}[\cite{AgricolaDileo20}]\label{prop:sasaki_einstein}
    A 3-$(\alpha,\delta)$-Sasaki manifold of dimension $4d+3$ is Riemannian Einstein if and only if $\delta = \alpha$ or $\delta = (2d+3)\alpha$. The respective Einstein constants $\Lambda$, satisfying $\Ric = \Lambda g$, evaluate to:
    \begin{enumerate}
        \item For $\delta = \alpha$ (the classical 3-$\alpha$-Sasaki metric):
        \[ \Lambda_1 = 2\alpha^2(2d+1). \]
        \item For $\delta = (2d+3)\alpha$ (the second, or squashed, Einstein metric):
        \[ \Lambda_2 = 2\alpha^2(4d^2+14d+9). \]
    \end{enumerate}
    In the 7-dimensional case $(d=1)$, the second Einstein metric is induced by a proper nearly parallel $G_2$-structure on the manifold. 

    The scalar curvature of a $3$-$(\alpha,\delta)$-Sasaki manifold is given by $$\scal=2(8d(d+2)\alpha\delta-6d\alpha^2+3\delta^2).$$
\end{prop}

The metrics of these manifolds form a 2-parameter family. In the explicit Lie algebraic construction of homogeneous 3-$(\alpha, \delta)$-Sasaki manifolds \cite[Thm.\ 3.1.1]{AgricolaDileoStecker21}, the Riemannian metric is expressed directly as a scaled sum of the Killing form on the horizontal and vertical subspaces. Translating this geometric structure to our setting, the metric $g$ corresponds exactly to a canonical variation of the bi-invariant metric $Q$. Specifically, the parameters $(\alpha, \delta)$ dictate the scaling factors via $t_0 = (2\alpha\delta)^{-1}$ and $t_1 = \delta^{-2}$, yielding the invariant metric $g = t_0 Q|_{\mathfrak{m}} \oplus t_1 Q|_{\mathfrak{su}(2)}$. Substituting these parameters into Theorem \ref{thm:operational_branching}, the Laplace-Beltrami eigenvalues for any such manifold are given explicitly by:
\begin{align}\label{eq:sasaki_spectrum_formula}
	\eta(\varrho_G, \varrho_H) &= 2\alpha\delta \left( c_{Q}(\varrho_{G}) - c_{Q\myrestriction\su(2)}(\varrho_{H}) \right) + \delta^2 c_{Q\myrestriction\su(2)}(\varrho_{H}) \nonumber \\
    &= 2\alpha\delta\cdot c_{Q}(\varrho_{G}) + \delta(\delta-2\alpha)\cdot c_{Q\myrestriction\su(2)}(\varrho_{H}).
\end{align}
Here, $\varrho_G \in \widehat{G}$ and $\varrho_H \in \widehat{\SU(2)}$ must satisfy the branching condition \eqref{eq:branching_condition}. 

Our goal in the following subsections is to apply our operational branching method to explicitly determine the admissible pairs $(\varrho_G, \varrho_H)$ for the non-exceptional classical series associated with the Lie groups $\SU(n+1)$, $\SO(m)$, and $\Sp(n+1)$, respectively. The relevant groups determining the fibrations are summarized in Table \ref{table:3sasaki_groups}.
\begin{table}[h]
		\centering
	\caption{Classical homogeneous 3-Sasaki data}
	\label{table:3sasaki_groups}
		\begin{tabular}{cccc}
			\toprule
			Type& $K$ & \(K\cdot \SU(2)\) & \(G\) \\
			\midrule
			$A_n$& $S(\U(n-1)\times\U(1)\id_2) $ &\(S(\U(n-1)\times\U(2))\) & \(\SU(n+1)\) \\
			 $B_n, D_n$  & $\SO(m-4)\times \Sp(1) $ &\(\SO(m-4)\times \SO(4)\) & \(\SO(m)\) \\
			$C_n$ & $ \Sp(n-1)$ &\(\Sp(n-1)\times \Sp(1)\) & \(\Sp(n)\) \\
			$C_n$ & $ \Sp(n-1)\times \Z_{2}$ &\(\Sp(n-1)\times \Sp(1)\) & \(\Sp(n)\) \\
			\bottomrule
		\end{tabular}
	\end{table}

		\subsection{$3$-$(\alpha,\delta)$-Sasaki manifolds of type $A_{n}$}
	We begin our analysis with the type $A_n$ family, corresponding to the special unitary group $G = \SU(n+1)$. We consider the standard maximal torus $T \subset \SU(n+1)$ consisting of diagonal matrices. Its Lie algebra $\t$ consists of diagonal matrices $\mathrm{diag}(ia_1, \dots, ia_{n+1})$ with $\sum_{k=1}^{n+1} a_k = 0$. The dual space $\t^*$ is spanned by the functionals $\lambda_k$, where $\lambda_k$ projects onto the $k$-th diagonal entry. Due to the trace condition, they satisfy $\sum_{k=1}^{n+1} \lambda_k = 0$.
	
	An irreducible representation of $\SU(n+1)$ is uniquely determined by its highest weight. Any dominant integral weight can be written in the form $\lambda = \sum_{k=1}^n z_k \lambda_k$ for a unique set of integers $z_k$ satisfying the dominance condition $z_1 \ge z_2 \ge \dots \ge z_n \ge 0$. We will denote such a representation by $\varrho(z_1, \dots, z_n) \in \widehat{\SU(n+1)}$.
	
    According to Table \ref{table:3sasaki_groups}, the branching described in Theorem \ref{thm:operational_branching}, requires the branching of representations of $\SU(n+1)$ with respect to the subgroup $K \cdot H \cong S(\U(n-1) \times \U(2))$. A direct analysis of this restriction is complicated by the special unitary condition, which couples the determinants of the two block factors. We employ a lifting strategy: we extend the representations from $\SU(n+1)$ to $\U(n+1)$, where the subgroup acts as a direct product $\U(n-1) \times \U(2)$. This allows us to apply the reciprocity lemma introduced in Section \ref{sec:the_tools}, effectively transforming the branching problem into a tensor product decomposition. The conditions for spherical representations are then derived using Proposition \ref{prop: tensor product decomposition}.
      
\begin{prop}\label{prop: Al spherical reps}
The irreducible representation $\varrho(z_{1},\dots,z_{n})\otimes \varrho(z_{n+1})\in \widehat{\SU(n+1)\otimes \SU(2)}$ is $S(U(n-1)\times U(1)\id_2)\times \Delta \SU(2)$-spherical 
if and only if 

\begin{enumerate}
\item Case $n=2$:
\[
z_1 + z_2 \equiv 0 \pmod{3},
\]
and $z_{3}$ must be even with 
\[
|z_1-2z_2| \le \frac{3z_{3}}{2} \le \min(2z_1-z_2,z_1+z_2). 
\]

\item Case $n=3$: 
\[
            z_1 + z_2 + z_3 \equiv 0 \pmod{4},\qquad 
            z_3 \le \frac{z_1 + z_2 + z_3}{4} \le z_2,
        \]
        and $z_{4}$ must be even satisfying
        \[
            |z_1-z_2-z_3| \le z_{4} \le z_1-z_2+z_3. 
        \]
\item Case $n\geq 4$: 
\[
z_1 + z_2 + z_n = 4z_i,\text{ for } i=3,\dots,n-1,
\]
and $z_{n+1}$ must be even with 
\[
|z_1-z_2-z_n| \le z_{n+1} \le z_1-z_2+z_n. 
\]
\end{enumerate}
In all cases the multiplicity of the invariant subspace is one. 
Note that, in all cases, $z_{n+1}$ is even (which must be the case since the $3$-Sasaki manifold has fiber $\SO(3))$.

\end{prop}
\begin{proof}
According to Theorem \ref{thm:operational_branching}, the representation $\varrho_{\SU(n+1)}(z_{1},\dots,z_{n}) \otimes \varrho_{\SU(2)}(z_{n+1})$ contributes to the spectrum if and only if its restriction to the base isotropy group $L = S(\U(n-1) \times \U(2))$ contains the component $1_K \otimes \varrho_{\SU(2)}(z_{n+1})^*$. Since irreducible representations of $\SU(2)$ are self-dual, the required fiber component is simply $\varrho_{\SU(2)}(z_{n+1})$.

To analyze the branching behavior, we lift the representation \linebreak$\varrho_{\SU(n+1)}(z_{1},\dots,z_{n})$ to a representation $\varrho_{\U(n+1)}(z_{1},\dots,z_{n},0)$ of the unitary group. Its restriction to the subgroup $\U(n-1) \times \U(2)$ yields a decomposition into irreducible modules of the form:
\begin{align*}
    \varrho_{\U(n+1)}(z_{1},...,z_{n},0)\myrestriction_{\U(n-1)\times \U(2)}
    =\bigoplus_{\textbf{x},\textbf{y}} m(\textbf{x},\textbf{y})\, \varrho_{\U(n-1)}(\textbf{x})\otimes \varrho_{\U(2)}(\textbf{y}),
\end{align*}
where $\textbf{x},\textbf{y}$ are suitable  integer tuples and $m(\textbf{x},\textbf{y})$ multiplicity coefficients.
We must determine which of these components descend to the required $1_K \otimes \varrho_{\SU(2)}(z_{n+1})$ structure under the group $L$. Elements of $L = S(\U(n-1) \times \U(2))$ can be parameterized by pairs $(A, B)$ with $A \in \U(n-1)$ and $B \in \SU(2)$, embedded into $\SU(n+1)$ via:
\[
    (A, B) \longmapsto \begin{pmatrix} A & 0 \\ 0 & \det(A)^{-1/2}B \end{pmatrix}.
\]
Such an element acts on a tensor product component $V_{\varrho_{\U(n-1)}(\mathbf{x})} \otimes V_{\varrho_{\U(2)}(\mathbf{y})}$ (from the $\U(n+1)$ restriction) as:
\begin{align*}
    \varrho_{\U(n-1)}(\mathbf{x})(A) \otimes \varrho_{\U(2)}(\mathbf{y})\left(\det(A)^{-1/2}B\right).
\end{align*}
Using the property that $\varrho_{\U(2)}(\mathbf{y})$ with $\textbf{y}=(y_1,y_2)$  acts on central scalars $c$ as multiplication by $c^{y_1+y_2}$, and that $\varrho_{\SU(2)}(y_1-y_2)$ is the standard restriction of the $\U(2)$-representation to $\SU(2)$ (see Proposition \ref{prop: branching GoWa}), we can factor out the determinant part:
\begin{align*}
    \left( \varrho_{\U(n-1)}(\mathbf{x})(A) \cdot \det(A)^{-(y_1+y_2)/2} \right) \otimes \varrho_{\SU(2)}(y_1-y_2)(B).
\end{align*}
 For this to match our target component $1_K \otimes \varrho_{\SU(2)}(z_{n+1})$, two independent constraints on the weights $\mathbf{x}$ and $\mathbf{y}$ must be satisfied:

First, the $K$-part (the action of $A)$ must be trivial. This forces $\varrho_{\U(n-1)}(\mathbf{x})$ to be a one-dimensional representation of the form $\det(A)^k$. Consequently, the weights must be constant, $x_1 = \dots = x_{n-1} = k$. To cancel the additional factor $\det(A)^{-(y_1+y_2)/2}$, the integer $k$ must satisfy the relation $k = (y_1+y_2)/2$.

Second, the $\SU(2)$-part of the action must match the fiber representation. This occurs if and only if their highest weights coincide: $y_1 - y_2 = z_{n+1}$.

Combining these conditions, the required component appears if and only if the specific module 
\[
\varrho_{\U(n-1)}(k^{n-1}) \otimes \varrho_{\U(2)}(y_1, y_2)
\]
occurs with non-zero multiplicity in the restriction of $\varrho_{\U(n+1)}(z_1, \dots, z_n, 0)$, where $y_1, y_2$ are determined by the relations 
\[
\frac{y_1+y_2}{2}=k\quad\text{and}\quad y_1-y_2= z_{n+1}.
\] In particular, this implies $z_{n+1}$ must be even.

Using the reciprocity lemma (Lemma \ref{lem:reciprocity}) and Proposition \ref{prop: tensor product decomposition}, this condition is equivalent to requiring that the tensor product multiplicity satisfies:
\[
\mathrm{mult}
\left(
\varrho_{\GL(n+1)}\big(z_{1},z_{2},(k^{n-3}),z_{n},0\big), 
\varrho_{\GL(n+1)}(k^{n-1})\otimes \varrho_{\GL(n+1)}(y_1,y_2)
\right)
\ge 1,
\]
for some $k\ge0$ and $y_1\ge y_2$ satisfying the above identities. For $n\geq 3$, it follows from Proposition \ref{prop: tensor product decomposition} that this occurs if and only if
\[
z_1=k+i_1,\qquad 
z_2=k+i_2,\qquad
z_n=i_3,\qquad 
0=i_4,
\]
with $i_1+i_2+i_3=2k$. Substituting the values for $i_1, i_2, i_3$ into the inequalities of the theorem, yields:
\begin{align*}
y_1-k & \le i_1 \le y_1, \\
y_1-i_1 & \le i_3 \le k,\\
0 & \le i_2 \le y_1-i_3.
\end{align*}
Since $y_1=k+\frac{z_{n+1}}{2}$, these become $z_1+z_2+z_n=4k$ and
\begin{align*}
\frac{z_{n+1}}{2} & \le z_1-k \le k+\frac{z_{n+1}}{2}, \\
2k+\frac{z_{n+1}}{2}-z_1 & \le z_n \le k,\\
0 & \le z_2-k \le k+\frac{z_{n+1}}{2}-z_n.
\end{align*}
These are the same as
\begin{align*}
z_{n+1}+2k & \le 2z_1 \le 4k+z_{n+1}, \\
4k+z_{n+1}-2z_1 & \le 2z_n \le 2k,\\
2k & \le 2z_2 \le 4k+z_{n+1}-2z_n.
\end{align*}
We replace $4k=z_1+z_2+z_n$ and we have
\begin{align*}
3z_1-z_2-z_n-2z_{n+1}&\ge 0, \\
-z_1+z_2+z_n+z_{n+1}&\ge 0, \\
z_1- z_2 + z_n - z_{n+1}& \ge 0, \\
z_1+z_2-3z_n &\ge 0, \\
-z_1+3z_2-z_n &\ge 0, \\
z_1-z_2-z_n+z_{n+1}& \ge 0.
\end{align*} 
Note that the first inequality follows from adding twice the third inequality to the fourth, and is therefore redundant (this corresponds to the fact that $y_1-k\le i_1$ follows from 
$y_1-i_1\le i_3\le k)$. 
Using the ordering relations $z_1\geq z_2\geq ...\geq z_n$, it is straightforward to see that these conditions are equivalent to those in the statement of the theorem. \\
For the case $n=2$, the reciprocity condition simplifies to the branching rule for $\GL(3) \downarrow \GL(2) \times \GL(1)$. According to \cite[Thm. 8.1.2]{Goodman Wallach}, the restriction of an irreducible representation $\varrho_{\GL(3)}(z_1, z_2, 0)$ decomposes into a sum of representations $\varrho_{\GL(2)}(y_1, y_2) \otimes \det^x$ where the weights $(y_1, y_2)$ interlace with $(z_1, z_2, 0)$, i.e.,
\[
z_1 \ge y_1 \ge z_2 \ge y_2 \ge 0,
\]
and the $\GL(1)$ weight is given by $x =  (z_1+z_2) - (y_1+y_2)$.
We identify $x$ with the parameter $k$ of the one-dimensional representation $\varrho_{\GL(1)}(k)$. Imposing the spherical conditions $y_1+y_2=2k$ and $y_1-y_2=z_3$ implies $z_1+z_2 = k + 2k = 3k$, and substituting $y_{1,2} = k \pm z_3/2$ into the interlacing inequalities yields the result of Case (1).
\end{proof}

\begin{rem}
    We remark that the classification of spherical representations in Proposition \ref{prop: Al spherical reps} can alternatively be derived by applying the explicit branching rules for unitary groups described in \cite{BenHalima07}. This approach involves lifting the $\SU(n+1)$-representations to $\U(n+1)$ and evaluating the interlacing conditions for the restriction to $\U(n-1) \times \U(2)$. While this method yields the same set of conditions, the evaluation of the resulting combinatorial constraints is quite involved and less transparent from a representation-theoretic perspective.
\end{rem}    
We can now explicitly state the spectrum.
     \begin{rem}
    While evaluating Casimir eigenvalues for classical Lie groups via Freudenthal's formula is standard (see e.g. \cite{Yamaguchi79}), we carry out these computations explicitly to carefully track the metric normalizations, scaling parameters, and weight conventions of our setup.
\end{rem}
\begin{thm}\label{thm: spec Al 3 alpha delta sasaki}
The spectrum of the Laplace-Beltrami operator on $(\SU(n+1)/S(\U(n-1)\times \U(1)\id_2),g_{t_{0},t_{1}})$ consists of the eigenvalues $\eta$ with multiplicities $\mult(\eta)$ given as follows:

\begin{enumerate}
    \item \emph{Case $n=2$:} 
    The eigenvalues depend on the parameters $z_1, z_2, z_3$ and are given by
    \[
    \eta(z_1,z_2,z_3)=4\alpha\delta\big( (z_1 + 1)^2  + z_2^2   - \tfrac{1}{3}(z_1 + z_2)^2 - 1 \big) + \delta(\delta - 2\alpha) z_{3} (z_{3}+2).
    \]
    The parameters $z_1, z_2 \ge 0$ are integers, and $z_3 \ge 0$ is even with
    \[
    z_1 + z_2 \equiv 0 \pmod{3}, \qquad |z_1-2z_2| \le \frac{3z_{3}}{2} \le \min(2z_1-z_2,z_1+z_2). 
    \]
    The multiplicity is given by
    \[
    \mult(\eta(z_1,z_2,z_3)) = \frac{1}{2}(z_{3}+1)(z_1 - z_2 + 1)(z_1 + 2)(z_2 + 1).
    \]

    \item \emph{Case $n \ge 3$:} 
    The eigenvalues depend only on the boundary variables $z_1, z_2, z_n$ and the fiber weight $z_{n+1}$:
       \begin{align*}
    &\eta(z_1,z_2,z_n,z_{n+1}) \\
    &= 4 \alpha \delta \Big( (z_1 + n-1)^2 + (z_2 + n-2)^2 + z_n^2   - \big(\tfrac{z_1 + z_2 + z_n}{2} + n-2\big)^2 - (n-1)^2 \Big) \\
    &\quad +\delta(\delta - 2\alpha) z_{n+1} (z_{n+1}+2).
    \end{align*}
    The parameters $z_1, z_2, z_n \ge 0$ are integers and $z_{n+1} \ge 0$ is even. They satisfy
     \[
      z_1\ge z_2,\qquad  z_1 + z_2 + z_n \equiv 0 \pmod{4}, \qquad z_n \le \frac{z_1 + z_2 + z_n}{4} \le z_2,
        \]
        and
        \[
        |z_1-z_2-z_n| \le z_{n+1} \le z_1-z_2+z_n. 
        \]
    The multiplicity is given by
\begin{align*}
    &\mult (\eta(z_{1}, z_2, z_n, z_{n+1})) \\[2mm]
    &= \frac{1}{n(n-1)^2(n-2)^3} (z_{n+1}+1)(z_1+n)(z_2+n-1)(z_n+1)\\[1mm]
    &\cdot (z_1-z_2+1)(z_1-z_n+n-1)(z_2-z_n+n-2) \\[1mm]
    &\cdot 
    \binom{\tfrac{z_1+z_2+z_n}{4}+n-2}{n-3} 
    \binom{\tfrac{3z_1-z_2-z_n}{4}+n-2}{n-3}
    \binom{\tfrac{-z_1+3z_2-z_n}{4}+n-3}{n-3}
    \binom{\tfrac{z_1+z_2-3z_n}{4}+n-3}{n-3}.
\end{align*}
\end{enumerate}
\end{thm} 
	\begin{proof}
		The highest weight is given by $\lambda=\sum_{i=1}^{n}z_{i}\lambda_{i}$, where $\lambda_{i}\in \t^{*}$ denotes the projection to the $i$-th diagonal entry. 
        Note that $\lambda_{i}$ is the dual element of $X_{i}-\frac{1}{n+1}\sum_{j=1}^{n+1}X_{j}\in \su(n+1)$, where $X_{j}\in \u(n+1)$ has only the value $i=\sqrt{-1}$ on the $j$-th diagonal entry. With respect to the standard Euclidean product given by $\langle A,B\rangle=-\tr(A\cdot B)$, the vectors $X_{i}$ yield an ONB and we see that
		\begin{align*}
			&\langle \lambda_{i},\lambda_{j}\rangle=\left\langle X_{i}-\frac{1}{n+1}\sum_{k=1}^{n+1}X_{k},X_{j}-\frac{1}{n+1}\sum_{k=1}^{n+1}X_{k}\right\rangle=\delta_{i,j}-\frac{2}{n+1}+\frac{n+1}{(n+1)^{2}}\\
			&=\delta_{i,j}-\frac{1}{n+1}.
		\end{align*}
		The half of the sum of positive roots is given by
		\begin{align*}
			\delta=\frac{1}{2}\sum_{1\leq i\leq j\leq n+1}\lambda_{i}-\lambda_{j}.
		\end{align*}
		Any fixed $\lambda_{k}$ occurs with a positive sign if $i=k$ and $k<j$, e.g. $n+1-k$-times and with a negative sign if $j=k$ and $i<k$, e.g. in total we have
		\begin{align*}
			\delta=\sum_{k=1}^{n+1}\frac{n+2-2k}{2}\lambda_{k}.
		\end{align*} 
		Using that $\lambda_{n+1}=-\sum_{k=1}^{n}\lambda_{k}$, we obtain
		\begin{align*}
			\delta=\sum_{k=1}^{n}(n+1-k)\lambda_{k}.
		\end{align*}
		The Freudenthal-formula evaluates to
		\begin{align*}
			& \left\langle \sum_{k=1}^{n} (2(n+1)+z_k-2k)\lambda_k, \sum_{j=1}^{n} z_j \lambda_j \right\rangle \\
			&= \sum_{i,j=1}^{n} z_j (2(n+1)+z_i-2i) \left(\delta_{ij}-\frac{1}{n+1}\right) 
		\end{align*}
		which further simplifies to
		\begin{align*}
			&= \sum_{k=1}^{n} z_k (2(n+1)+z_k-2k) - \frac{1}{n+1} \left( \sum_{j=1}^{n} z_j \right) \left( \sum_{k=1}^{n} (2(n+1)+z_k-2k) \right) \\
			&= \sum_{k=1}^{n} (2(n+1)z_k + z_k^2 - 2kz_k) - \frac{1}{n+1} \left( \sum_{j=1}^{n} z_j \right) \left( n(n+1) + \sum_{k=1}^{n} z_k \right) \\
			&= \sum_{k=1}^{n} z_k^2 + \sum_{k=1}^{n}(n+2-2k) z_k - \frac{1}{n+1}\left(\sum_{k=1}^{n}z_k\right)^2.	
		\end{align*}
		The unique positive root which is at the same time a dominant integral element is $\lambda_{1}-\lambda_{n+1}$ which is the root of $\su(2)\subset \su(n+1)$. We compute the Freudenthal formula with respect to the highest weight $(\lambda_{1}-\lambda_{n+1})\frac{z_{n+1}}{2}$ which belongs to an irreducible $\su(2)$-representation:
		\begin{align*}
			&\left\langle (\lambda_{1}-\lambda_{n+1})(1+\frac{z_{n+1}}{2}),(\lambda_{1}-\lambda_{n+1})\frac{z_{n+1}}{2}\right\rangle=\frac{z_{n+1}}{2}\cdot(1+\frac{z_{n+1}}{2})\left\langle \lambda_{1}-\lambda_{n+1},\lambda_{1}-\lambda_{n+1}\right\rangle\\
			&=\frac{z_{n+1}}{2}\cdot(1+\frac{z_{n+1}}{2})\left(1-\frac{1}{n+1}+\frac{2}{n+1}+1-\frac{1}{n+1}\right)=\frac{z_{n+1}\cdot(2+z_{n+1})}{2}.
		\end{align*}
		Recall from \cite{GoertschesRoschigStecker23} that the $3$-Sasaki metric on the orthogonal complement of $\su(2)^{*}$ in $\su(n+1)^{*}$ is given by $g=-4(d+2)B$, where $d=(\dim(G/K)-3)/4$ and $B=-\frac{1}{2(n+1)}\langle \cdot,\cdot\rangle$ is the form on the dual space induced by the Killing form of $\su(n+1)$.
		Let us translate $d$ to $n$:
        \begin{align*}
        \dim (G)=\dim (\SU(n+1))=n^2+2n,\quad
\dim (K)=\dim (S(\U(n-1)\times \U(1)\id_2))=(n-1)^2
        \end{align*}
        and hence
        \begin{align*}
			d&=\frac{\dim(G/K)-3}{4}=\frac{n^{2}+2n-n^{2}+2n-1-3}{4}=n-1,
		\end{align*}
		so the metric on the dual space is given in the parallel case $1=\delta^{2}=2\alpha\delta$ by
		\begin{align*}
			\frac{4(n+1)}{2(n+1)}\langle\cdot,\cdot\rangle=2\langle\cdot,\cdot\rangle.
		\end{align*}
		In order to obtain the spectrum, we only need to rescale the respective Freudenthal formulas
		\begin{align*}
			&2\alpha\delta\cdot\left(2\left(\sum_{k=1}^{n} z_k^2 + \sum_{k=1}^{n}(n+2-2k) z_k - \frac{1}{n+1}\left(\sum_{k=1}^{n}z_k\right)^2\right)-z_{n+1}\cdot(2+z_{n+1})\right)\\
			&+\delta^{2}\left(z_{n+1}\cdot(2+z_{n+1})\right)
		\end{align*}
        By rearranging the terms, we can easily separate the variables. The terms depending on $z_{n+1}$ factorize directly to $\delta(\delta-2\alpha)z_{n+1}(z_{n+1}+2)$. For the remaining terms, we rely on the constraints established in Proposition \ref{prop: Al spherical reps}.

For the case $n=2$, the remaining part of the eigenvalue simplifies to $4\alpha\delta(z_1^2 + z_2^2 + 2z_1 - \frac{1}{3}(z_1+z_2)^2)$. By completing the square for $z_1$, this can be compactly written as $4\alpha\delta((z_1+1)^2 + z_2^2 - \frac{1}{3}(z_1+z_2)^2 - 1)$. 

For $n \ge 3$, Proposition \ref{prop: Al spherical reps} states that $z_k = \frac{z_1+z_2+z_n}{4}$ for all $3 \le k \le n-1$ (note that for $n=3$ this condition is empty, but the formula below still applies without modification). Let $S = z_1+z_2+z_n$. Summing over all $n$ components yields $\sum_{k=1}^n z_k = S + (n-3)\frac{S}{4} = \frac{n+1}{4}S$. Consequently, the subtracted squared sum becomes $\frac{1}{n+1}(\sum_{k=1}^n z_k)^2 = \frac{n+1}{16}S^2$. 
Evaluating the linear and quadratic terms for the middle block $3 \le k \le n-1$ yields:
\[
    \sum_{k=3}^{n-1} z_k^2 = (n-3)\frac{S^2}{16} \quad \text{and} \quad \sum_{k=3}^{n-1} (n+2-2k) z_k = \frac{S}{4} \sum_{k=3}^{n-1} (n+2-2k) = 0,
\]
because the coefficients \(n+2-2k\) sum to zero over this range. Grouping the remaining $S^2$ terms gives $(n-3)\frac{S^2}{16} - \frac{n+1}{16}S^2 = -\frac{1}{4}S^2$. Finally, evaluating the remaining $k \in \{1,2,n\}$ terms and completing the squares yields the simplified expression. Thus, the spectrum is given by:
\begin{itemize}
    \item For $n=2$: 
    \[
    \eta(z_1,z_2,z_3)=4\alpha\delta\big( (z_1 + 1)^2  + z_2^2   - \tfrac{1}{3}(z_1 + z_2)^2 - 1 \big) + \delta(\delta - 2\alpha) z_{3} (z_{3}+2).
    \]
    \item For $n\geq 3$:
    \begin{align*}
    &\eta(z_1,z_2,z_n,z_{n+1}) \\
    &= 4 \alpha \delta \Big( (z_1 + n-1)^2 + (z_2 + n-2)^2 + z_n^2   - \big(\tfrac{z_1 + z_2 + z_n}{2} + n-2\big)^2 - (n-1)^2 \Big) \\
    &\quad +\delta(\delta - 2\alpha) z_{n+1} (z_{n+1}+2).
    \end{align*}
\end{itemize}

		Finally, we compute the multiplicity of each eigenvalue. As shown in the proof of Proposition \ref{prop: Al spherical reps}, the multiplicity of the trivial representation in the branching is one. Therefore, the total multiplicity of the eigenvalue $\eta(z_{1},\dots,z_{n+1})$ is the product of the dimensions of the corresponding irreducible representations of $\SU(n+1)$ and $\SU(2)$. The dimension formula can be found in \cite[Thm. 6.3]{Fulton_Harris} and can be written as 
        \[
\frac{1}{n!}\, (z_{n+1} + 1) (z_n + 1) (z_{n-1} + 2) \dots (z_1 + n) \cdot \prod_{1\le i < j \le n}  \frac{z_i - z_j + j - i}{j - i}.
\]
Using the constraints from Proposition \ref{prop: Al spherical reps}, we can factorize this expression into a much more explicit form. 
For $n=2$, directly evaluating the product over $1 \le i < j \le 2$ immediately yields the dimension:
\[
    \mult (\eta(z_{1},z_{2},z_{3})) = \frac{1}{2} (z_{3}+1)(z_1 - z_2 + 1)(z_1 + 2)(z_2 + 1).
\]

For $n \ge 3$, Proposition \ref{prop: Al spherical reps} forces the intermediate variables in the block $K = \{3, \dots, n-1\}$ to be strictly equal to $\frac{S}{4} = \frac{z_1+z_2+z_n}{4}$. Consequently, any term in the product $\prod_{i<j} \frac{z_i-z_j+j-i}{j-i}$ with both indices in $K$ reduces to $1$. 

The remaining contributions involving $K$ arise from four distinct sources, each forming a product of $(n-3)$ factors:
\begin{enumerate}
    \item The cross-terms where $i=1$ and $j \in K$.
    \item The cross-terms where $i=2$ and $j \in K$.
    \item The cross-terms where $i \in K$ and $j=n$.
    \item The corresponding terms $(z_k + n + 1 - k)$ for $k \in K$ originating from the prefactor of the dimension formula.
\end{enumerate}
Since the product of any $m$ consecutive integers can be written as $m! \binom{X}{m}$, we can express these four products as four binomial coefficients of the form $\binom{\dots}{n-3}$, extracting a factor of $((n-3)!)^4$ in the numerator.
We can then collect all corresponding denominator terms $(j-i)$ for these cross-terms, which evaluate to $(n-2)!$, $(n-3)!$, and $(n-3)!$, respectively. Dividing our extracted numerator by these factorials, the isolated terms for $i,j \in \{1, 2, n\}$, and the global $\frac{1}{n!}$ yields the rational constant:
\[
    \frac{((n-3)!)^4}{n! \cdot (n-2)! \cdot ((n-3)!)^2 \cdot (n-1)(n-2)} = \frac{1}{n(n-1)^2(n-2)^3}.
\]
The total multiplicity for $n \ge 3$ simplifies to:
\begin{align*}
    &\mult (\eta(z_{1}, z_2, z_n, z_{n+1})) \\[2mm]
    &= \frac{1}{n(n-1)^2(n-2)^3} (z_{n+1}+1)(z_1+n)(z_2+n-1)(z_n+1)\\[1mm]
    &\cdot (z_1-z_2+1)(z_1-z_n+n-1)(z_2-z_n+n-2) \\[1mm]
    &\cdot 
    \binom{\tfrac{z_1+z_2+z_n}{4}+n-2}{n-3} 
    \binom{\tfrac{3z_1-z_2-z_n}{4}+n-2}{n-3}
    \binom{\tfrac{-z_1+3z_2-z_n}{4}+n-3}{n-3}
    \binom{\tfrac{z_1+z_2-3z_n}{4}+n-3}{n-3}.
\end{align*}
 \end{proof}
\begin{rem}
	The case $n=2$ of our general result for the type $A_l$ family corresponds to the 3-$(\alpha$,$\delta)$-Sasaki structures on the Aloff-Wallach manifold $W^{1,1} = \SU(3)/S^1$. In this specialization, both the conditions on the highest weight parameters $(z_1, z_2, z_3)$ given in Proposition \ref{prop: Al spherical reps} and the resulting eigenvalue formula exactly reproduce the results of \cite[Thm. 3.7]{AgricolaHenkel25}, which were derived specifically for this manifold. This serves as a crucial consistency check for our generalized formula.
\end{rem}

\subsection{$3$-$(\alpha,\delta)$-Sasaki manifolds of type $B_{n}$ and $D_{n}$}
We now consider the families associated with the special orthogonal groups $G = \SO(m)$ for $m \ge 5$. Let $n = \lfloor m/2 \rfloor$ denote the rank of the group. The case $m = 2n+1$ corresponds to the Type $B_n$ series, while $m = 2n$ corresponds to the Type $D_n$ series. 
 
 In both cases, the maximal torus has a Lie algebra $\mathfrak{t} \cong \R^n$, which is generated by $n$ block-diagonal matrices of the form $\begin{pmatrix} 0 & -x_j \\ x_j & 0 \end{pmatrix}$. The dual basis for $\mathfrak{t}^*$ is $\{\lambda_1, \dots, \lambda_n\}$, where $\lambda_{j}$ denotes the projection onto the coefficient $x_{j}$ of the $j$-th block.
 
 An irreducible representation is uniquely determined by its highest weight $\Lambda = \sum_{i=1}^{n} z_i \lambda_i$, where the coefficients $z_i$ are integers satisfying dominance conditions that depend on the type of the group:
 \begin{itemize}
 	\item For type $B_{n}$ $(m = 2n+1)$, the coefficients satisfy $z_1 \ge z_2 \ge \dots \ge z_n \ge 0$.
 	\item For type $D_{n}$ $(m = 2n)$, the coefficients satisfy $z_1 \ge z_2 \ge \dots \ge z_{n-1} \ge |z_n|$.
 \end{itemize}
 In particular, for the type $D_{n}$ case, the last coefficient $z_n$ may be a negative integer. We denote the irreducible representation corresponding to the highest weight $(z_1, \dots, z_n)$ by $\varrho_{\SO(m)}(z_1, \dots, z_n)$.

	\begin{lem}\label{lem: SO(4) vs SU(2)SU(2)}
		An irreducible representation $\varrho(a)\otimes\varrho(b)\in \hat{\SU(2)\times\SU(2)}$ descends (by the double cover $\SU(2)\times \SU(2)\rightarrow \SO(4))$ to an irreducible representation of $\SO(4)$  if and only if $a+b$ is even. The corresponding $\SO(4)$-representation has highest weight \begin{align*}
			\frac{a+b}{2}\cdot\lambda_{1}+\frac{a-b}{2}\cdot\lambda_{2}.
		\end{align*}
	\end{lem}
	\begin{proof}
		The double cover $\SU(2)\times \SU(2)\rightarrow\SO(4)$ is obtained by identifying $\R^{4}\cong \mathbb{H}$ and using quaternionic multiplication of the unit quaternions $S^{3}\cong\nolinebreak \SU(2)$ from the left and from the right
		\begin{align*}
			\pi:\SU(2)\times \SU(2)\rightarrow \SO(4),\quad (u,v)\mapsto\:( q\in \mathbb{H}\mapsto u\cdot q\cdot v^{-1}).
		\end{align*}
		The maximal torus of $\SU(2)\times \SU(2)$ is given by $\{(e^{it_{1}},e^{it_{2}})\mid t_{1},t_{2}\in \R\}$. It acts by a rotation in the $(1,i)$-plane by $e^{i(t_{1}-t_{2})}$ and by a rotation in the $(j,k)$-plane by $e^{i(t_{1}+t_{2})}$. Hence, the weights transform to the claimed form (up to permuting the diagonal blocks). The product of homogeneous polynomials over $\C^{2}$ of degree $a$ and $b$ is invariant under $\ker\pi=\pm(\id,\id)$ if and only if $a+b$ is even. 
	\end{proof}
	\begin{prop}\label{prop: bl dl spherical reps}
		The irreducible representation $\varrho(z_{1},...,z_{n})\otimes \varrho(z_{n+1})\in \hat{\SO(m)\times \SU(2)} $ is $\SO(m-4)\times \SU(2)\times \Delta\SU(2)$-spherical if and only if
		\begin{enumerate}
			\item $z_{n+1} \equiv 0 \mod 2$,
			\item $z_{i}=0$ for all $i=3,\dots,n$, and
			\item $z_{2}=\frac{z_{n+1}}{2}$ $(m=5)$ or $z_{2}\ge \frac{z_{n+1}}{2}$ $(m\ge 6)$.
		\end{enumerate}
		In those cases, the trivial representation of 
		$\mathrm{SO}(m-4) \times \mathrm{SU}(2) \times \Delta\mathrm{SU}(2)$ appears with multiplicity given by
		\[
		\text{(*)}\;
		\begin{cases}
			1,&\text{if $m=5$;} \\
			\dfrac{z_1-z_2+1}{m-5}\dbinom{z_1+m-5-\frac{z_{n+1}}{2}}{m-6}\dbinom{z_2+m-6-\frac{z_{n+1}}{2}}{m-6}
			,&\text{if $m\ge 6$.}
		\end{cases}
		\]
		This applies for $m=2n$ and $m=2n+1$.
	\end{prop}
	\begin{proof}
		We place $\SU(2)$ upper left and $\SO(m-4)$ lower right. The embedding $\SU(2)\rightarrow \SO(4)$ is given by the action of unit quaternions from the \emph{left} on the quaternions, see \cite[Proof of Prop 9.1]{GoertschesRoschigStecker23} . The branching rules for $K\cdot H=\SO(4)\times \SO(m-4)$ can be found in \cite[Thm. 10]{Chami} if $m$ is even and in \cite[Thm. 11]{Chami} if $m$ is odd. Although \cite{Chami} stated an if and only if condition, one of those conditions is very tricky to verify. However, we investigate the following two necessary conditions stated in \cite{Chami}: If the irreducible representation $\varrho_{\SO(4)}(x_{1},x_{2})\otimes \varrho_{\SO(m-4)}(x_{3},...,x_{n})$ occurs in $\varrho_{\SO(m)}(z_{1},...,z_{n})$, it needs to satisfy
		\begin{align*}
			|z_{i+1}|\leq |x_{i}|\leq z_{i-1},\quad \text{for}\quad 2\leq i\leq n-1,\quad \text{ and }|x_{n}|\leq z_{n-1}.
		\end{align*}
        According to Theorem \ref{thm:operational_branching}, the representation $\varrho_{\SO(m)}(z_{1},...,z_{n})\otimes \varrho_{\SU(2)}(z_{n+1})$ contributes to the spectrum if and only if the restriction of $\varrho_{\SO(m)}(z_{1},...,z_{n})$ to the base isotropy group $K \cdot H \cong \SO(m-4) \times \SO(4)$ contains a component of the form $1_K \otimes \varrho_{\SO(4)}(x_1, x_2)$, where the $\SO(4)$-representation restricts to the dual of the fiber representation $\varrho_{\SU(2)}(z_{n+1})$. Since representations of $\SU(2)$ are self-dual, the target fiber component is simply $\varrho_{\SU(2)}(z_{n+1})$.
        For the $K$-part $(\SO(m-4))$ to be trivial, we must have $x_3 = \dots = x_n = 0$. By the branching rule stated above, this is equivalent to
 $$ x_3=\dots=x_n=0 \quad \Rightarrow \quad z_3=\dots=z_n=0. $$
 We further focus on the representation $\varrho_{\SO(4)}(x_{1},x_{2})$ of the remaining $\SO(4)$ block. By Lemma \ref{lem: SO(4) vs SU(2)SU(2)}, it corresponds to the $\SU(2)\times\SU(2)$-representation
 $$\hat{\SO(4)}\ni\varrho_{\SO(4)}(x_{1},x_{2})\cong \varrho_{\SU(2)}(x_{1}+x_{2})\otimes \varrho_{\SU(2)}(x_{1}-x_{2})\in \hat{\SU(2)\times\SU(2)}.$$
 As the Lie group generated by the Reeb vector fields acts from the right on the quaternions (see \cite[Proof of Prop. 9.1]{GoertschesRoschigStecker23}), the $H$-factor embeds into $\SO(4)$ as the right $\SU(2)$-factor. For this component to match our target $1 \otimes \varrho_{\SU(2)}(z_{n+1})$, the left $\SU(2)$-factor must act trivially, meaning $x_1 + x_2 = 0$, so $x_1 = -x_2$. Simultaneously, the right factor must match the fiber representation, yielding $x_1 - x_2 = z_{n+1}$. Combining these conditions gives $z_{n+1} = 2x_1$ which implies that  $x_1 = z_{n+1}/2$ and $x_2 = -z_{n+1}/2$. 
       The interlacing condition $z_2 \ge |x_2|$ from the $\SO(5) \downarrow \SO(4)$ branching (see Proposition \ref{prop: branching GoWa})  then implies $z_2 \ge z_{n+1}/2$. In summary, for a representation $\varrho(z_{1},...,z_{n})\otimes \varrho(z_{n+1})$ to be spherical, it is necessary that
		\begin{align*}
			z_{3}=\dots=z_{n}=0, \quad z_{n+1} \text{ is even, and } \quad z_2 \ge z_{n+1}/2.
		\end{align*}
		
		 We prove that those conditions are sufficient by induction and by using the standard branching rules which are summarized in Proposition \ref{prop: branching GoWa}:\\ \noindent
		 \textbf{Case $m=5$:} the branching from $\mathrm{SO}(5)$ to $\mathrm{SO}(4)$ says that 
		 \[
		 \varrho_{\mathrm{SO}(5)}(z_1,z_2)=\bigoplus \varrho_{\mathrm{SO}(4)}(x_1,x_2)
		 \]
		 where the sum is multiplicity free and runs over all $x_1,x_2$ satisfying 
		 \[
		 z_1\ge x_1\ge z_2 \ge |x_2|.
		 \]
		 Thus, $x_1,x_2=x,-x$ appears in the sum if and only if $x=z_2$ (which is condition (3) above), 
		 in which case it does with multiplicity 1.

		 \

		 \noindent
		 \textbf{Case $m=6$:} the branching from $\mathrm{SO}(6)$ to $\mathrm{SO}(5)$ says that 
		 \[
		 \varrho_{\mathrm{SO}(6)}(z_1,z_2,0)=\bigoplus \varrho_{\mathrm{SO}(5)}(y_1,y_2)
		 \]
		 where the sum is multiplicity free and runs over all $y_1,y_2$ satisfying 
		 \[
		 z_1\ge y_1\ge z_2 \ge y_2 \ge |0|.
		 \]
		 Now,  case $m=5$ says that 
		 $\varrho_{\mathrm{SO}(4)}(x,-x)$ will appear in the above sum for all such pairs $y_1,y_2$ with $x=y_2$ (with multiplicity 1  in each summand). 
		 
		 Therefore, $\text{mult}\big(\varrho_{\mathrm{SO}(4)}(x,-x),\varrho_{\mathrm{SO}(6)}(z_1,z_2,0)\big)>0$
		 if and only if $z_2\ge x$ (which is condition (3) above), and in this case
		 \begin{align*}
		 	\text{mult}\big(\varrho_{\mathrm{SO}(4)}(x,-x),\varrho_{\mathrm{SO}(6)}(z_1,z_2,0)\big) 
		 	& = |\{\text{pairs }y_1,y_2:  z_1\ge y_1\ge z_2\text{ and } x=y_2\}| \\
		 	& = z_1-z_2+1,
		 \end{align*}
		 which coincides with (*).

		 \

		 \noindent
		 \textbf{Case $m=7$:} the branching from $\mathrm{SO}(7)$ to $\mathrm{SO}(6)$ says that 
		 \[
		 \varrho_{\mathrm{SO}(7)}(z_1,z_2,0)=\bigoplus \varrho_{\mathrm{SO}(6)}(y_1,y_2,y_3)
		 \]
		 where the sum is multiplicity free and runs over all $y_1,y_2,y_3$ satisfying 
		 \[
		 z_1\ge y_1\ge z_2 \ge y_2 \ge 0\ge |y_3|.
		 \]
		 Thus
		 \[
		 \varrho_{\mathrm{SO}(7)}(z_1,z_2,0)=\bigoplus \varrho_{\mathrm{SO}(6)}(y_1,y_2,0)
		 \]
		 where the sum is multiplicity free and runs over all $y_1,y_2$ satisfying 
		 \[
		 z_1\ge y_1\ge z_2 \ge y_2 \ge 0.
		 \]
		 Now,  case $m=6$ says that 
		 $\varrho_{\mathrm{SO}(4)}(x,-x)$ will appear in the above sum for all such pairs $y_1,y_2$ satisfying $y_2\ge x$, 
		 with multiplicity $y_1-y_2+1$ in each summand. 
		 
		 Therefore, $\text{mult}(\varrho_{\mathrm{SO}(4)}(x,-x),\varrho_{\mathrm{SO}(7)}(z_1,z_2,0))>0$
		 if and only if $z_2\ge x$ (which is condition (3) above), and in this case
		 \begin{align*}
		 	\text{mult}(\varrho_{\mathrm{SO}(4)}(x,-x),\varrho_{\mathrm{SO}(7)}(z_1,z_2,0))
		 	& =\sum_{y_1=z_2}^{z_1}\sum_{y_2=x}^{z_2}y_1-y_2+1 \\
		 	& =\frac12(z_1-z_2+1)(z_1-x+2)(z_2-x+1),
		 \end{align*}
		 which coincides with (*).

		 \

		 \noindent
		 \textbf{Case $m$ arbitrary:} 
		 The same argument used above yields
		 \[
		 \varrho_{\mathrm{SO}(m)}(z_1,z_2)=\bigoplus \varrho_{\mathrm{SO}(m-1)}(y_1,y_2)
		 \]
		 where the sum is multiplicity free and runs over all $y_1,y_2$ satisfying 
		 \[
		 z_1\ge y_1\ge z_2 \ge y_2 \ge 0.
		 \]
		 By induction hypothesis, 
		 $\text{mult}\big(\varrho_{\mathrm{SO}(4)}(x,-x),\varrho_{\mathrm{SO}(m-1)}(y_1,y_2)\big)>0$
		 if and only if $y_2\ge x$ (which implies $z_2\geq x$ by interlacing), and in this case
		 \begin{align*}
		 	\text{mult}\big(\varrho_{\mathrm{SO}(4)}(x,-x),\varrho_{\mathrm{SO}(m)}(z_1,z_2)\big)
		 	=\sum_{y_1=z_2}^{z_1}\sum_{y_2=x}^{z_2}\text{mult}\big(\varrho_{\mathrm{SO}(4)}(x,-x),\varrho_{\mathrm{SO}(m-1)}(y_1,y_2)\big). 
		 \end{align*}
		 It follows by induction that 
		 \[
		 \text{mult}\big(\varrho_{\mathrm{SO}(4)}(x,-x),\varrho_{\mathrm{SO}(m)}(z_1,z_2)\big)=
		 \frac{z_1-z_2+1}{m-5}\binom{z_1+m-5-x}{m-6}\binom{z_2+m-6-x}{m-6}
		 \]
		 satisfies the recursive identity above.
	\end{proof}
	\begin{thm}\label{thm: spectrum type B,D}
		The spectrum of $(\SO(m)/(\SO(m-4)\times \SU(2))),g_{t_{0},t_{1}})$ is given by the collection of numbers 
		\begin{align*}			\eta(z_{1},z_2,z_{n+1})=4\alpha\delta \cdot \left({z_{1}(z_{1}+m-2)+z_{2}(z_{2}+m-4)}\right)+ \delta(\delta-2\alpha)\cdot\left({z^{2}_{n+1}+2z_{n+1}}\right),
		\end{align*}
		where the parameters $z_1, z_2, z_{n+1} \ge 0$ are integers satisfying
\begin{align*}
	z_1 \ge z_2, \quad z_{n+1} \equiv 0 \pmod 2 \quad \text{and} \quad
	\begin{cases}
		z_2 = z_{n+1}/2, & \text{if } m=5; \\[1mm]
		z_2 \ge z_{n+1}/2, & \text{if } m \ge 6.
	\end{cases}
\end{align*}
The multiplicities are given by
\begin{enumerate}
    \item Case $m=5$:
    \begin{align*}
      \mult (\eta(z_{1},z_2,2z_2)) =  \frac{1}{6} (z_1 + z_2 + 2)(z_1 - z_2 + 1)(2z_1 + 3)(2z_2 + 1)^2,
    \end{align*}
    \item Case $m\geq 6$:
     \begin{align*}
        \mult (\eta(z_{1},z_2,z_{n+1})) 
        &= \frac{(z_{n+1}+1)\cdot (z_1+z_2 + m-3)(z_1-z_2 + 1)}{(m-2)(m-3)(m-4)^2}\;  \\[2mm]
        &\hspace{1cm} \cdot (2z_1+m-2)(2z_2+m-4) 
        \binom{z_1+m-4}{m-5}\binom{z_2+m-5}{m-5}\\
        &\hspace{1cm}\cdot \dfrac{z_1-z_2+1}{m-5}\dbinom{z_1+m-5-\frac{z_{n+1}}{2}}{m-6}\dbinom{z_2+m-6-\frac{z_{n+1}}{2}}{m-6}.
    \end{align*}
\end{enumerate}
 
	\end{thm}

	\begin{proof}
		Note that with respect to the Euclidean inner product $\langle A,B\rangle=-\tr(A\cdot B)$ the metrical dual of the projection $\lambda_{i}$ defined in the beginning of this section, is given by $E_{i}$, where $$E_{i}=\begin{pmatrix}
			0& -\frac{1}{2}\\
			\frac{1}{2}&0
		\end{pmatrix}$$ being the $i$-th diagonal block. Hence, we have that
		\begin{align*}
			\langle \lambda_{i},\lambda_{j}\rangle=\frac{\delta_{ij}}{2}.
		\end{align*} 
		The positive roots are given by
		\begin{align*}
			\{\lambda_{i}\pm \lambda_{j}\:\mid \: 1\leq i<j\leq n\}\cup \{\lambda_{i}\:\mid \: 1\leq i\leq n\},\quad \text{$m=2n+1$ is odd, i.e. type $B_{n}$}\\
			\{\lambda_{i}\pm \lambda_{j}\:\mid \: 1\leq i<j\leq n\},\quad 1\leq i<j\leq n,\quad \text{$m=2n$ is even, i.e. type $D_{n}$.}
		\end{align*}
		We obtain that
		\begin{align*}
			2\rho_{B_{n}}=\sum_{j=1}^{n}\lambda_{j}+\sum_{1\leq i<j\leq n}\lambda_{i}-\lambda_{j}+\sum_{1\leq i<j\leq n}\lambda_{i}+\lambda_{j}=\sum_{k=1}^{n}(2n-2k+1)\lambda_{k}\\
			2\rho_{D_{n}}=\sum_{1\leq i<j\leq n}\lambda_{i}-\lambda_{j}+\sum_{1\leq i<j\leq n}\lambda_{i}+\lambda_{j}=\sum_{k=1}^{n-1}2(n-k)\lambda_{k}.
		\end{align*}
		We compute the Freudenthal formula for $\varrho_{\SO(m)}(z_{1},z_{2})$. We start with respect to $B_{n}$:
		\begin{align*}
			\langle (z_{1}+2n-1)\lambda_{1}+(z_{2}+2n-3)\lambda_{2},z_{1}\lambda_{1}+z_{2}\lambda_{2}\rangle=\frac{z_{1}(z_{1}+2n-1)+z_{2}(z_{2}+2n-3)}{2}
		\end{align*}
		and get with respect to $D_{n}$:
		\begin{align*}
			\langle (z_{1}+2n-2)\lambda_{1}+(z_{2}+2n-4)\lambda_{2},z_{1}\lambda_{1}+z_{2}\lambda_{2}\rangle=\frac{z_{1}(z_{1}+2n-2)+z_{2}(z_{2}+2n-4)}{2}.
		\end{align*}
		As $m=2n+1$ in case of $B_{n}$ and $m=2n$ in the case of $D_{n}$, we can summarize the formulas into
		\begin{align*}
			\langle z_{1}\lambda_{1}+z_{2}\lambda_{2}+2\rho,z_{1}\lambda_{1}+z_{2}\lambda_{2}\rangle=\frac{z_{1}(z_{1}+m-2)+z_{2}(z_{2}+m-4)}{2}.
		\end{align*}
		The unique maximal root     
        which lies in the fundamental Weyl chamber is in both cases the root $ \lambda_1+\lambda_2$ which is the root of $\su(2)\subset \so(m)$. We compute its Casimir constant:
		\begin{align*}
			\left\langle (\frac{z_{n+1}}{2}+1)(\lambda_{1}+\lambda_{2}),\frac{z_{n+1}}{2}(\lambda_{1}+\lambda_{2})\right\rangle= \frac{\frac{z_{n+1}^{2}}{2}+z_{n+1}}{2}=\frac{z^{2}_{n+1}+2z_{n+1}}{4}.
		\end{align*}
		Recall from \cite{GoertschesRoschigStecker23} that the $3$-Sasaki metric on the orthogonal complement of $\su(2)^{*}$ in $\so(m)^{*}$ is given by $g=-4(d+2)B$, where $d=(\dim(G/K)-3)/4$ and $B=-\frac{1}{(m-2)}\langle \cdot,\cdot\rangle$ is the form on the dual space induced by the Killing form. Let us translate $d$ to $m$: 
		\begin{align*}
			d=\frac{\dim((\SO(m)/(\SO(m-4)\times \SU(2)))-3}{4}=\frac{(4m - 13)-3}{4}=\frac{4m-16}{4}=m-4.
		\end{align*}
		In total we have that
		\begin{align*}
			g=-4(m-2)B=\frac{4(m-2)}{m-2}\langle\cdot,\cdot\rangle=4 \langle\cdot,\cdot\rangle
		\end{align*}
		and observe that $\varrho_{\SO(m)}(z_{1},z_{2})\otimes \varrho_{\SU(2)}(z_{n+1})$ produces with respect to the $3$-$(\alpha,\delta)$-Sasaki metric the eigenvalue:
		\begin{align*}
			2\alpha\delta \left(2({z_{1}(z_{1}+m-2)+z_{2}(z_{2}+m-4)})-({z^{2}_{n+1}+2z_{n+1}})\right)+ \delta^{2}\left({z^{2}_{n+1}+2z_{n+1}}\right).
		\end{align*}
        The dimension formula for irreducible $\SO(m)$ representations can be found in \cite[Eq. (24.29)]{Fulton_Harris} if $m$ is odd and in \cite[Eq. (24.41)]{Fulton_Harris} if $m$ is even. The dimension of $\varrho_{\SO(m)}(z_{1},z_{2})\otimes \varrho_{\SU(2)}(z_{n+1})$ is given by:
\begin{enumerate}
    \item Case $m=5$:
    \begin{align*}
         &(z_{3}+1) \cdot \frac{1}{6}(z_1 - z_2 + 1)(z_1 + z_2 + 2)(2z_1 + 3)(2z_2 + 1).
    \end{align*}

    \item Case $m=6$:
    \begin{align*}
         &(z_{4}+1) \cdot \frac{1}{12}(z_1 - z_2 + 1)(z_1 + z_2 + 3)(z_1 + 2)^2 (z_2 + 1)^2.
    \end{align*}

    \item Case $m=2n \ge 8$:
    We define $l_i = z_i + n - i$ and $m_i = n - i$. The dimension formula is given by:
    \begin{align*}
         &(z_{n+1}+1) \cdot \frac{l_1^2 - l_2^2}{m_1^2 - m_2^2} \cdot \prod_{j=3}^{n} \frac{(l_1^2 - m_j^2)(l_2^2 - m_j^2)}{(m_1^2 - m_j^2)(m_2^2 - m_j^2)}.
    \end{align*}

    \item Case $m=2n+1 \ge 7$:
    We define $l_i = z_i + n - i + \frac{1}{2}$ and $m_i = n - i + \frac{1}{2}$. The dimension formula is given by:
    \begin{align*}
         & (z_{n+1}+1) \cdot \frac{l_1 l_2}{m_1 m_2} \cdot \frac{l_1^2 - l_2^2}{m_1^2 - m_2^2} \cdot \prod_{j=3}^{n} \frac{(l_1^2 - m_j^2)(l_2^2 - m_j^2)}{(m_1^2 - m_j^2)(m_2^2 - m_j^2)}.
    \end{align*}
\end{enumerate}
        Remarkably, the formulas for the even and odd cases are the same and can also be written as follows:
    \begin{align*}
        &\frac{(z_{n+1}+1) \cdot (z_1+z_2 + m-3)(z_1-z_2 + 1)}{(m-2)(m-3)(m-4)^2}\;  \\[1mm]
        &\hspace{1cm} \cdot (2z_1+m-2)(2z_2+m-4) 
        \binom{z_1+m-4}{m-5}\binom{z_2+m-5}{m-5}.
    \end{align*}
 
Multiplying these dimension formulas with the multiplicity formulas of the trivial representation stated in Proposition \ref{prop: bl dl spherical reps}, yields the wished result.
	\end{proof}
	\subsection{$3$-$(\alpha,\delta)$-Sasaki manifolds of type $C_{n}$}
	We follow the parameterization of the highest weights for the Lie group $G = \operatorname{Sp}(n)$ as detailed in \cite[Sec. 2]{Chami_Spn}. The complexified Lie algebra of $G$ is denoted by $\mathfrak{g}$.
	A Cartan subalgebra $\mathfrak{t}$ of $\mathfrak{g}$ can be chosen as the set of diagonal matrices:
	$$ \mathfrak{t} = \{ \operatorname{diag}(x_1, \dots, x_n, -x_1, \dots, -x_n) : x_j \in \mathbb{C} \}.$$
	The elements $\lambda_j$ are considered as elements of the dual space $\mathfrak{t}^*$, forming a basis where $\lambda_j$ maps a diagonal matrix $H = \operatorname{diag}(x_1, \dots, x_n, -x_1, \dots, -x_n)$ to its $j$-th component $x_j$, i.e., $\lambda_j(H) = x_j$.
	
	The root system of $G$ with respect to $\mathfrak{t}$ is given by:
	$$ R = \{ \pm \lambda_i \pm \lambda_j : 1 \le i < j \le n \} \cup \{ \pm 2\lambda_i : 1 \le i \le n \}.$$
	The set of positive roots can be chosen as:
	$$ R^+ = \{ \lambda_i \pm \lambda_j : 1 \le i < j \le n \} \cup \{ 2\lambda_i : 1 \le i \le n \}.$$
	A set of simple roots for $G$ is:
	$$ \alpha_1 = \lambda_1 - \lambda_2, \alpha_2 = \lambda_2 - \lambda_3, \dots, \alpha_{n-1} = \lambda_{n-1} - \lambda_n, \alpha_n = 2\lambda_n.$$
	An irreducible representation of $G$ is uniquely determined by its highest weight $\Lambda$. Any dominant weight $\Lambda$ for $(\mathfrak{g}, \mathfrak{t})$ corresponding to an irreducible representation of $G$ has the form:
	$$ \Lambda = h_1\lambda_1 + h_2\lambda_2 + \dots + h_n\lambda_n, $$
	where the coefficients $h_i$ are integers $(h_i \in \mathbb{Z})$ satisfying the dominance conditions:
	$$ h_1 \ge h_2 \ge \dots \ge h_n \ge 0. $$
	
	\subsubsection{The subgroup $K\cdot H = Sp(1) \times Sp(n-1)$}
	We consider the subgroup $$K\cdot H = \operatorname{Sp}(1) \times \operatorname{Sp}(n-1)\subset G=\operatorname{Sp}(n).$$ The complexified Lie algebra of $K\cdot H$ is $ \mathfrak{sp}(1)_{\mathbb{C}} \oplus \mathfrak{sp}(n-1)_{\mathbb{C}}$.
	The Cartan subalgebra $\mathfrak{t}$ of $\mathfrak{g}$ contains a Cartan subalgebra for $\mathfrak{sp}(1)_{\mathbb{C}} \oplus \mathfrak{sp}(n-1)_{\mathbb{C}}$. Specifically, we can identify:
	\begin{itemize}
		\item $\mathfrak{t}_1 = \{ \operatorname{diag}(x_1, 0, \dots, 0, -x_1, 0, \dots, 0) : x_1 \in \mathbb{C} \}$ as a Cartan subalgebra for the $\mathfrak{sp}(1)_{\mathbb{C}}$ factor. The corresponding functional is $\lambda_1$.
		\item $\mathfrak{t}_{n-1} = \{ \operatorname{diag}(0, x_2, \dots, x_n, 0, -x_2, \dots, -x_n) : x_j \in \mathbb{C} \text{ for } j=2,\dots,n \}$ as a Cartan subalgebra for the $\mathfrak{sp}(n-1)_{\mathbb{C}}$ factor. The corresponding functionals are $\lambda_2, \dots, \lambda_n$.
	\end{itemize}
	An irreducible representation of $K\cdot H$ is determined by a pair of highest weights, one for $\operatorname{Sp}(1)$ and one for $\operatorname{Sp}(n-1)$.
	Any dominant weight $\Lambda'$ for $\mathfrak{k}$ corresponding to an irreducible representation of $K$ can be written in terms of the basis $\{\lambda_1, \dots, \lambda_n\}$ of $\mathfrak{t}^*$ as:
	$$ \Lambda' = k_1\lambda_1 + k_2\lambda_2 + \dots + k_n\lambda_n, $$
	where the coefficients $k_i$ are integers $(k_i \in \mathbb{Z})$ satisfying the dominance conditions for each factor:
	\begin{itemize}
		\item For the $\operatorname{Sp}(1)$ factor (with highest weight $k_1\lambda_1)$: $k_1 \ge 0$.
		\item For the $\operatorname{Sp}(n-1)$ factor (with highest weight $k_2\lambda_2 + \dots + k_n\lambda_n$ relative to the basis $\lambda_2, \dots, \lambda_n)$: $k_2 \ge k_3 \ge \dots \ge k_n \ge 0$.
	\end{itemize}
	Thus, the conditions for $\Lambda'$ to be a dominant weight for $K = \operatorname{Sp}(1) \times \operatorname{Sp}(n-1)$ are:
	$$ k_1 \ge 0 \quad \text{and} \quad k_2 \ge k_3 \ge \dots \ge k_n \ge 0. $$
\begin{prop}\label{prop: cl spherical reps}
		The representation $\varrho(z_{1},..,z_{n})\otimes \varrho(z_{n+1})\in \hat{\Sp(n)}\times \hat{\Sp(1)}$ is $\Sp(n-1)\times\Delta \Sp(1)$-spherical if and only if
		\begin{align*}
			z_{3}=...=z_{n}=0,\quad z_{n+1}=z_{1}-z_{2}.
		\end{align*}
		Furthermore, it is spherical with respect to the quotient $(\Sp(n-1)\times \Z_{2})\times\Delta \Sp(1)$ if and only if it additionally satisfies the parity condition
        \begin{align*}
            z_1 \equiv z_2 \pmod 2.
        \end{align*}
		In both cases, the multiplicity is $1$.
	\end{prop}
	\begin{proof}
By Theorem \ref{thm:operational_branching}, the representation $\varrho_{\Sp(n)}(z_{1},...,z_{n})\otimes \varrho_{\Sp(1)}(z_{n+1})$ is spherical if and only if the restriction of $\varrho_{\Sp(n)}(z_{1},...,z_{n})$ to the base isotropy group $K \cdot H = \Sp(n-1) \times \Sp(1)$ contains the component $1_K \otimes \varrho_{\Sp(1)}(z_{n+1})^*$. Because irreducible representations of $\Sp(1) \cong \SU(2)$ are self-dual, and the $K$-part must be trivial, this is precisely equivalent to requiring that the component 
 $$ \varrho_{\Sp(n-1)}(0) \otimes \varrho_{\Sp(1)}(z_{n+1}) $$
 occurs in the restriction.    
        The branching $\Sp(n)\downarrow \Sp(n-1)\times\Sp(1)$ is classical (see for example
        \cite{Knapp01}, \cite{WallachYacobi09}, \cite{Chami_Spn}). Following \cite[Thm. 4.1 (a)]{Knapp01}, a necessary and sufficient condition for the subspace of $\Sp(n-1)$-invariants in the irreducible representation $\varrho(z_1, \dots, z_n)$ to be non-zero is that the highest weight components vanish for indices strictly greater than $2$. This yields the condition
\begin{align*}
    z_3 = \dots = z_n = 0.
\end{align*}
Furthermore, \cite[Thm. 4.1 (b)]{Knapp01} states that the resulting representation on the remaining factor $\Sp(1)$ is equivalent to the restriction of the irreducible representation of $U(2)$ with highest weight $(z_1, z_2)$ to the subgroup $\Sp(1)\cong \SU(2) \subset U(2)$.
The restriction of an irreducible representation of $U(2)$ with highest weight $(z_1, z_2)$ to $\SU(2)$ is the irreducible representation with highest weight $z_1 - z_2$.
Consequently, the representation $\varrho_{\Sp(n-1)}(0)\otimes\varrho_{\Sp(1)}(z_{n+1})  $ occurs in $\varrho_{\Sp(n)}(z_1, \dots, z_n)$ if and only if
\begin{align*}
    z_{n+1} = z_1 - z_2.
    \end{align*}
    The multiplicity is $1$.   

    For the quotient space $\Sp(n)/(\Sp(n-1)\times \Z_{2})$, the representation must additionally act trivially on the central $\Z_2 \subset \Sp(1)$ of the fiber. This holds if and only if the highest weight $z_{n+1}$ of the $\Sp(1)$-representation is even (which corresponds to representations factoring through $\SO(3))$.
	\end{proof}
	\begin{rem}
		This result is consistent with \cite[Thm. 8.1.5]{Goodman Wallach} from which we can directly conclude that $\varrho(z_{1},...,z_{n})$ is $\Sp(n-1)$-spherical if and only if $z_{3}=...=z_{n}=0$. Additionally, \cite[Thm. 8.1.5]{Goodman Wallach} states that the corresponding multiplicity is $z_1-z_2+1$ which is consistent with the $\Sp(1)$ representation 
$\varrho_{\Sp(1)}(z_1-z_2)$ stated in Proposition \ref{prop: cl spherical reps}.
\end{rem}

	\begin{thm}\label{thm: spectrum 3 sasaki type C}
    The spectrum of the Laplace-Beltrami operator on the $3$-$(\alpha,\delta)$-Sasaki manifolds $S^{4n-1} \cong \Sp(n)/\Sp(n-1)$ and $\mathbb{R}P^{4n-1} \cong \Sp(n)/(\Sp(n-1)\times \Z_2)$ for $n\geq 2$ equipped with $g_{t_{0},t_{1}}$ is given by the collection of numbers
		\begin{align*}
			\eta(z_{1},z_{2})=2\alpha\delta(2z_{1}(n-1)+2z_{1}z_{2}+2nz_{2})+\delta^{2}((z_{1}-z_{2})(z_{1}-z_{2}+2)),
		\end{align*}
		where the parameters $z_{1}\geq z_{2}\geq 0$ are integers. For the projective space $\mathbb{R}P^{4n-1}$, the parameters must additionally satisfy the parity condition 
        \[
            z_1 \equiv z_2 \pmod 2.
        \]
        The multiplicity of the eigenvalue $\eta(z_1,z_2)$ in both cases can be expressed as 
        \begin{align*}
        \mult (\eta(z_{1},z_{2})) 
        = \frac{(z_1-z_2+1)^2(z_1+z_2 + 2n-1)}{(2n-1)(2n-2)} \binom{z_1+2n-2}{2n-3}\binom{z_2+2n-3}{2n-3}.
    \end{align*}       
	\end{thm}
	\begin{proof}
		Note that with respect to the Euclidean inner product $\langle A,B\rangle=-\tr(A\cdot B)$ the metrical dual of the projection $\lambda_{i}$ to the $i$-th diagonal entry is given by $\frac{1}{2}E_{i}$, where $E_{i}$ has the value $1$ on the $i$-th and $-1$ on the $i+n$-th diagonal entry. Hence, we have that
		\begin{align*}
			\langle \lambda_{i},\lambda_{j}\rangle=\frac{\delta_{ij}}{2}.
		\end{align*} 
		The half of sum of positive roots is given by
		\begin{align*}
			\rho&=\frac{1}{2}\left(\sum_{i=1}^{n}2\lambda_{i}+\sum_{1\leq i<j\leq n}\lambda_{i}+ \lambda_{j}+\sum_{1\leq i<j\leq n}\lambda_{i}- \lambda_{j}\right).
		\end{align*}
		We fix a $1\leq k\leq n$ and count how often $\lambda_{k}$ occurs here: In the first summand we have it $2$ times, in the second it occurs $(n-1)$ times and in the third it occurs $(n-k)-(k-1)$ times (we get a plus for $k=i$ and a minus for $j=k)$. In total we get
		\begin{align*}
			\rho=\sum_{k=1}^{n}(n+1-k)\lambda_{k}.	
		\end{align*}
		We compute the Freudenthal formula for $\varrho_{\Sp(n)}(z_{1},z_{2})$:
		\begin{align*}
			\langle z_{1}\lambda_{1}+z_{2}\lambda_{2}+2\sum_{k=1}^{n}(n+1-k)\lambda_{k},z_{1}\lambda_{1}+z_{2}\lambda_{2}\rangle=\frac{z_{1}^{2}+z_{2}^{2}+2z_{1}n+2(n-1)z_{2}}{2}.
		\end{align*}
		Analogously, we compute the Freudenthal formula for the $\su(2)\subset \sp(n)$ whose root is $2\lambda_{1}$, as this is the unique maximal root in the fundamental Weyl-chamber, see \cite{GoertschesRoschigStecker23}. 
		\begin{align*}
			\langle \lambda_{1}z_{n+1} +2\lambda_{1},\lambda_{1}z_{n+1}\rangle=\frac{z_{n+1}^{2}+2z_{n+1}}{2}.
		\end{align*}
		Recall from \cite{GoertschesRoschigStecker23} that the $3$-Sasaki metric on the orthogonal complement of $\su(2)^{*}$ in $\sp(n)^{*}$ is given by $g=-4(d+2)B$, where $d=(\dim(G/K)-3)/4$ and $B=-\frac{1}{2(n+1)}\langle \cdot,\cdot\rangle$ is the form on the dual space induced by the Killing form of $\sp(n)$. Let us translate $d$ to $n$:
		\begin{align*}
			d=\frac{\dim(G/K)-3}{4}=\frac{(4n-1)-3}{4}=\frac{4n-4}{4}=n-1.
		\end{align*}
		In total we have that
		\begin{align*}
			g=-4(d+2)B=\frac{4(n+1)}{2(n+1)}\langle\cdot,\cdot\rangle=2 \langle\cdot,\cdot\rangle
		\end{align*}
		and observe that $\varrho_{\Sp(n)}(z_{1},z_{2})\otimes \varrho(z_{1}-z_{2})$ produces with respect to the $3$-$(\alpha,\delta)$-Sasaki metric the eigenvalue
		\begin{align*}
			&2\alpha\cdot\delta\left((z_{1}^{2}+z_{2}^{2}+2z_{1}n+2(n-1)z_{2})-((z_{1}-z_{2})^{2}+2(z_{1}-z_{2}))\right)
            \\ &+\delta^{2} ((z_{1}-z_{2})^{2}+2(z_{1}-z_{2})).
		\end{align*}
		This can be simplified to
		\begin{align*}
			2\alpha\delta(2z_{1}(n-1)+2z_{1}z_{2}+2nz_{2})+\delta^{2}((z_{1}-z_{2})(z_{1}-z_{2}+2)).
		\end{align*}	
        Finally, we compute the multiplicity of each eigenvalue. Since the branching is multiplicity-free, the total multiplicity is the product of the dimensions of the corresponding irreducible representations of $\Sp(n)$ and $\Sp(1)$. The dimension of the $\Sp(1) \cong \SU(2)$ representation $\varrho(z_1-z_2)$ is simply $z_1-z_2+1$.
The dimension of the irreducible $\Sp(n)$-representation $\varrho_{\Sp(n)}(z_1, z_2)$ with highest weight $\lambda = z_1\lambda_1 + z_2\lambda_2$ can computed using \cite[Eq. (24.19)]{Fulton_Harris}.
	\end{proof}
\begin{rem}[Comparison to \cite{2BLP22}]\label{rem: comparison to bettiol}
	Our results for the type C family can be compared to the spectrum of the isotropic quaternionic Hopf fibration, denoted by the metric $h(t,t,t)$ in \cite{2BLP22}. It is crucial to note that the parameter $n$ is used differently in the two contexts. In our work, for the space $\Sp(n)/\Sp(n-1)$, $n$ denotes the rank of the group $G=\Sp(n)$, and the manifold has dimension $4n-1$. In \cite{2BLP22}, the space is denoted $S^{4n+3} = \Sp(n+1)/\Sp(n)$, therefore, we denote their parameter $n$ with the auxiliary parameter $\tilde{n}$, which corresponds to our $n-1$.
	
	The eigenvalues in \cite[Thm. 3.8]{2BLP22} are parameterized by integers $k,l$ satisfying $k \ge l \ge 0$ and $k \equiv l \pmod 2$. Their formula, written with the parameter $\tilde{n}$, is:
	\begin{align*}
		\eta_{\text{BLP22b}}(k,l) = k(k + 4\tilde{n} + 2) + l(l + 2) \left( \frac{1}{t^2} - 1 \right).
	\end{align*}
	To match our setting, we substitute $\tilde{n} = n-1$:
	\begin{align*}
		\eta_{\text{BLP22b}}(k,l) &= k(k + 4(n-1) + 2) + l(l + 2) \left( \frac{1}{t^2} - 1 \right) \\
		&= k(k + 4n - 2) + l(l + 2) \left( \frac{1}{t^2} - 1 \right) \\
		&= (k^2 + 4nk - 2k - l^2 - 2l) + \frac{1}{t^2} l(l+2).
	\end{align*}
	
	Now, we take our eigenvalue formula for the type C family,
	\begin{align*}
		\eta(z_1,z_2) = 2\alpha\delta(2z_{1}(n-1)+2z_{1}z_{2}+2nz_{2})+\delta^{2}((z_{1}-z_{2})(z_{1}-z_{2}+2)),
	\end{align*}
	and specialize it to the corresponding one-parameter family. This requires setting $2\alpha\delta = 2$ and identifying $\delta^2 = 1/t^2$ to align the metric normalizations. We then express our formula in terms of the parameters $k=z_1+z_2$ and $l=z_1-z_2$. The term with $\delta^2$ becomes $\frac{1}{t^2}l(l+2)$. The constant term becomes:
	\begin{align*}
		& 2 \left( 2z_{1}(n-1)+2z_{1}z_{2}+2nz_{2} \right) \\
		&= 2 \left( (z_1+z_2)(n-1) + z_1(z_2 - (n-1)) + z_2(n+z_1) \right) \\ 
		&= k^2 + 4nk - 2k - l^2 - 2l.
	\end{align*}
	Thus, our formula under this specialization is:
	\begin{align*}
		\eta(k,l) = (k^2 + 4nk - 2k - l^2 - 2l) + \frac{1}{t^2} l(l+2).
	\end{align*}
	The two expressions are identical. This confirms that our general formula for the 3-$(\alpha$,$\delta)$-Sasaki spectrum correctly reproduces the known spectrum of the quaternionic Berger spheres.
	
	Moreover, the multiplicities match exactly. In \cite[Thm. 3.8]{1BLP22}, the total multiplicity for the pair $(k,l)$ is given by $(l+1)d_{p,q}$, where $p=(k+l)/2$ and $q=(k-l)/2$. Under our identification, $p=z_1$ and $q=z_2$. Substituting $\tilde{n}=n-1$ into their formula \cite[Eq. (3.12)]{1BLP22} yields:
	\begin{align*}
		(l+1)d_{z_1,z_2} = (z_1 - z_2 + 1) \frac{(z_1+z_2+2n-1)(z_1-z_2+1)}{(2n-1)(z_1+1)} \binom{z_1+2n-2}{z_1} \binom{z_2+2n-3}{z_2}.
	\end{align*}
	By expanding the binomial coefficients and absorbing the $(z_1+1)$ factor into the denominator's factorial, this expression algebraically rearranges exactly to our simplified multiplicity formula:
	\begin{align*}
		\mult(\eta(z_1,z_2)) = \frac{(z_1-z_2+1)^2(z_1+z_2+2n-1)}{(2n-1)(2n-2)} \binom{z_1+2n-2}{2n-3} \binom{z_2+2n-3}{2n-3}.
	\end{align*}
\end{rem}
	
\section{Stiefel Manifolds}\label{sec:Stiefel Manifolds}
The methods developed in Section \ref{sec:the_tools} allow for the spectral analysis of invariant metrics on a broader class of homogeneous spaces. In this section, we focus on the Stiefel manifolds $V_2(\mathbb{F}^{n})$ of orthonormal $2$-frames in $\mathbb{F}^{n}$, where $\mathbb{F}$ denotes the field of real $(\mathbb{R})$ or complex $(\mathbb{C})$ numbers.

These manifolds naturally fiber over the Grassmannians $Gr_2(\mathbb{F}^{n})$ of $2$-planes. Let $G$ denote the standard compact Lie group $(\SO(n)$ for $\mathbb{F}=\mathbb{R}$ or $\SU(n)$ for $\mathbb{F}=\mathbb{C})$ acting transitively on $V_2(\mathbb{F}^{n})$. We fix a bi-invariant metric $Q$ on the Lie algebra $\mathfrak{g}$ of $G$, defined as the negative of the trace form:
\begin{align*}
    Q(X,Y) = -\tr(XY).
\end{align*}
The tangent space of $V_2(\mathbb{F}^{n}) \cong G/K$ is naturally identified with the orthogonal complement $\mathfrak{k}^\perp \subset \mathfrak{g}$. This space admits an $\Ad(K)$-invariant decomposition $\mathfrak{k}^\perp = \mathfrak{m} \oplus \mathfrak{h}$ with respect to $Q$, where $\mathfrak{m}$ is the horizontal subspace tangent to the Grassmannian and $\mathfrak{h}$ is the vertical subspace tangent to the fiber. We consider the family of $G$-invariant metrics $g_{t_0, t_1}$, the canonical variation, defined by uniformly scaling the vertical space:
 \begin{align}\label{eq:canonical_variation}
     g_{t_0, t_1} = t_0 \cdot Q\myrestriction_{\mathfrak{m}} + t_1 \cdot Q\myrestriction_{\mathfrak{h}}, \quad t_0, t_1 > 0.
 \end{align}
When $t_0 = t_1$, this recovers the normal homogeneous metric.

The geometry of these deformed metrics was studied extensively by Jensen \cite{Jensen73, Jensen75}. In particular, in \cite{Jensen75}, Jensen provided a geometric realization of these metrics by constructing equivariant embeddings of the Stiefel manifolds into higher-dimensional Grassmannians, which coincide with the canonical variation   \eqref{eq:canonical_variation} for the specific parameter range $t_1/t_0 \in (0, 1]$.

While the geometry of these metrics is well-understood, their explicit spectra are mostly unknown. For the standard metric $(t_0=t_1)$, the $\SO(n)$-equivariant decomposition of $L^2(V_k(\mathbb{R}^n))$ is classical. Gelbart \cite{Gelbart74} gives an explicit realization of the relevant $\SO(n)$-components by constructing "Stiefel harmonics" using polynomial restrictions. Very recently, Birtea et al. \cite{BirteaCasuComanescu25} derived explicit coordinate formulas for the Laplace-Beltrami operator on $V_k(\mathbb{R}^n)$ viewed as a constraint submanifold. However, these analytical and coordinate-based approaches do not easily extend to solving the eigenvalue problem for canonical variations.
 
From a representation-theoretic perspective, Tsukamoto \cite{Tsukamoto81} successfully computed the spectra of the Grassmannians $\SO(n+2)/(\SO(2)\times \SO(n))$ and $\Sp(n+1)/(\Sp(1)\times \Sp(n))$. Geometrically, these symmetric spaces appear precisely as the \emph{base manifolds} of our fibrations. Our work extends Tsukamoto's results by computing the spectrum of the \emph{total space} $V_2(\mathbb{F}^{n})$, which requires handling the additional representations introduced by the fiber.

\subsection{Complex Stiefel manifolds $V_2(\mathbb{C}^n)$}

The complex Stiefel manifold $V_2(\mathbb{C}^{n+1}) \cong \SU(n+1)/\SU(n-1)$ fibers over the complex Grassmannian $Gr_2(\mathbb{C}^{n+1})$ with fiber $H \cong \U(2)$. Topologically, these are Spin manifolds exactly when $n$ is even.

While the standard canonical variation \eqref{eq:canonical_variation} scales the vertical space $\mathfrak{h}$ uniformly, the algebraic structure of the complex fiber allows for finer deformations. The vertical algebra decomposes as $\mathfrak{h} \cong \mathfrak{u}(2) = \mathfrak{su}(2) \oplus \mathfrak{u}(1)$. As discussed by Arvanitoyeorgos, Sakane, and Statha \cite{ArvanitoyeorgosSakaneStatha20}, one may scale the center and the semisimple part of the fiber independently, leading to a 3-parameter family of metrics known as metrics of \textit{Jensen type}:
 \begin{align}\label{eq:complex_stiefel_metric_3param}
     g_{t_0, t_1, t_2} = t_0 \cdot Q\myrestriction_{\mathfrak{m}} + t_1 \cdot Q\myrestriction_{\mathfrak{su}(2)} + t_2 \cdot Q\myrestriction_{\mathfrak{u}(1)}.
 \end{align}

This parameterization constitutes a \emph{generalized} canonical variation. Following the construction in Section \ref{sec:spectral_theory}, this metric can be realized naturally reductively by enlarging the transitive group to $G' = \SU(n+1) \times \SU(2) \times \U(1)$. This geometric setup allows us to apply Theorem \ref{thm:operational_branching}, expressing the Laplace-Beltrami operator via the generalized Casimir operators of the individual factors.

\begin{rem}
    The existence of homogeneous Einstein metrics on $V_2(\mathbb{C}^{n+1})$ has been studied in \cite{ArvanitoyeorgosSakaneStatha20}. They investigated the full space of invariant metrics and found Einstein metrics within the family \eqref{eq:complex_stiefel_metric_3param} as well as non-Jensen type Einstein metrics for certain dimensions. The canonical variation defined in \eqref{eq:canonical_variation} does not seem to contain Einstein metrics (cf. \cite{Jensen75}, \cite{ArvanitoyeorgosSakaneStatha20}).
\end{rem}
\begin{prop}[{\cite[p.205]{ArvanitoyeorgosSakaneStatha20}}]\label{prop: Einstein complex Stiefel}
For any $n \ge 2$, the metric $g_{t_0, t_1, t_2}$ is Einstein if the parameters satisfy
\begin{align*}
     t_1 = t_0 \left[ 1 \pm \sqrt{\frac{n^2-2n}{n^2-1}} \right] ,\quad
 t_2 = t_0 \left[ \frac{n-1}{(n+1)(2n-1)} \left( 4n+1 \mp 3 \sqrt{\frac{n^2-2n}{n^2-1}} \right) \right]. 
\end{align*}
For any such choice of parameters, the corresponding Einstein constant is given by $$\Lambda=\frac{(n+1)t_2}{2 t_0^2}=\frac{n^2-1}{2 t_0 (n+1)(2n-1)} \left[ 4n+1 \mp 3 \sqrt{\frac{n^2-2n}{n^2-1}} \right].$$
\end{prop}
\begin{proof}
      In \cite[Sec. 6.4, p.\ 203]{ArvanitoyeorgosSakaneStatha20} invariant metrics on the complex Stiefel manifolds $V_{2m}(\mathbb{C}^{2m+\tilde{n}}) \cong \SU(2m+\tilde{n})/\SU(\tilde{n})$ are investigated. Setting $m=1$ (and thus $\ell=1)$ and $\tilde{n}=n-1$ precisely yields our manifold $V_2(\mathbb{C}^{n+1}) \cong \SU(n+1)/\SU(n-1)$ with dimension parameter $N = n+1$. 
    The authors utilize the negative of the Killing form $B$ on $\mathfrak{su}(n+1)$ as the bi-invariant background metric \cite[p.\ 166]{ArvanitoyeorgosSakaneStatha20}, which relates to our trace form $Q(X,Y) = -\tr(XY)$ via $B = 2(n+1)Q$.

    For their specific search of Einstein metrics on these spaces (cf.\ \cite[p.\ 203]{ArvanitoyeorgosSakaneStatha20}), they choose a parametrization of the invariant metric $g_{\text{ASS20}}$ where the relative scales of the submodules $\mathfrak{n}_7$ and $\mathfrak{n}_8$ are normalized to $x_{(7)} = x_{(8)} = 1$. In our fibration, these two submodules exactly span the horizontal space $\mathfrak{m}$. Matching this horizontal scale, our metric $g_{t_0, t_1, t_2}$ corresponds to the scaled metric
    \begin{align*}
        g_{t_0, t_1, t_2} = \frac{t_0}{2(n+1)} g_{\text{ASS20}},
    \end{align*}
    provided we identify the remaining fiber parameters appropriately. 
    
    Furthermore, to ensure a diagonal metric on the center of the isotropy algebra, they set the parameter variables $a=d=1$ and $b=c=0$ \cite[p.\ 203]{ArvanitoyeorgosSakaneStatha20}. The $\mathfrak{su}(2)$-part of the fiber is spanned by $\mathfrak{n}_6 \oplus \mathfrak{h}_5$. The authors recover metrics of \textit{Jensen's type} by enforcing that the scales of these two modules coincide, i.e., $v_5 = x_{(6)}$, which reduces their system to the algebraic condition $A(x_{(6)}) = 0$ \cite[p.\ 204]{ArvanitoyeorgosSakaneStatha20}. Thus, the relative scale of the entire $\mathfrak{su}(2)$-fiber is exactly $x_{(6)}$, yielding $t_1/t_0 = x_{(6)}$. The remaining 1-dimensional center $\mathfrak{u}(1)$ corresponds to their subspace $\mathfrak{h}_4$, whose scale is denoted by $v_4$, yielding $t_2/t_0 = v_4$.

     The explicit solutions for $x_{(6)}$ are provided in \cite[p.\ 205]{ArvanitoyeorgosSakaneStatha20}. Substituting $m=1$ and $\tilde{n}=n-1$ into these expressions yields $x_{(6)} = 1 \pm \sqrt{\frac{n^2-2n}{n^2-1}}$, which verifies our formula for $t_1/t_0$. The values for $t_2/t_0 = v_4$ follow directly from the corresponding explicit formulas provided alongside these solutions.

    Let $\lambda_{\text{ASS20}}$ be the Einstein constant of the metric $g_{\text{ASS20}}$. By \cite[Prop.\ 4.4, p.\ 173]{ArvanitoyeorgosSakaneStatha20}, the Ricci tensor component $r_4$ along the $\mathfrak{u}(1)$-center evaluates to 
    \begin{align*}
        r_4 = \frac{v_4}{4} \left( \frac{1}{x_{(6)}^2} \left[\begin{matrix} 4 \\ 6 6 \end{matrix}\right] + \frac{1}{x_{(7)}^2} \left[\begin{matrix} 4 \\ 7 7 \end{matrix}\right] + \frac{1}{x_{(8)}^2} \left[\begin{matrix} 4 \\ 8 8 \end{matrix}\right] \right),
    \end{align*}
    using their specific structure constants. Given $a=1$, $b=0$, and $\ell=m=1$, applying \cite[Lem.\ 4.3]{ArvanitoyeorgosSakaneStatha20} yields $\left[\begin{smallmatrix} 4 \\ 6 6 \end{smallmatrix}\right] = 0$ and $\left[\begin{smallmatrix} 4 \\ 7 7 \end{smallmatrix}\right] = \left[\begin{smallmatrix} 4 \\ 8 8 \end{smallmatrix}\right] = 1/2$. Since $x_{(7)} = x_{(8)} = 1$, the sum inside the parenthesis simplifies to exactly $1$, yielding the Einstein condition $\lambda_{\text{ASS20}} = r_4 = v_4/4$.

    Since the Ricci tensor is invariant under global scaling of the metric, $\Ric(g_{t_0, t_1, t_2}) = \Ric(g_{\text{ASS20}})$. Thus, the true Einstein constant $\Lambda$ associated with our metric $g_{t_0, t_1, t_2}$ scales inversely to the metric itself:
    \begin{align*}
        \Lambda = \Lambda_{\text{ASS20}} \frac{2(n+1)}{t_0} = \frac{v_4}{4} \frac{2(n+1)}{t_0} = \frac{t_2 / t_0}{4} \frac{2(n+1)}{t_0} = \frac{(n+1)t_2}{2 t_0^2}.
    \end{align*}
    Substituting the explicit formula for $t_2$ completes the proof.
\end{proof}
We parameterize the irreducible representations of the fiber group $H \cong \U(2)$ by tuples $(z_{n+1}, z_{n+2})$, where:
\begin{itemize}
    \item $z_{n+1} \in \mathbb{N}_0$ denotes the highest weight of the $\SU(2)$-part.
    \item $z_{n+2} \in \mathbb{Z}$ denotes the weight of the central $\U(1)$-part.
\end{itemize}
Note that for a representation to lift to $\U(2)$, the weights must satisfy the parity condition $z_{n+1} \equiv z_{n+2} \pmod 2$.
\begin{prop}\label{prop:spherical complex Stiefel}
    Let $n \ge 2$. The irreducible representation 
    \[
    \varrho_{\SU(n+1)}(z_1, \dots, z_n) \otimes \varrho_{\SU(2)}(z_{n+1})\otimes \varrho_{\U(1)}(z_{n+2})
    \]
    is $\SU(n-1)\times\Delta (\SU(2)\times \U(1))$-spherical if and only if $z_{n+1} \equiv z_{n+2} \pmod 2$, $z_{n+2} \le -z_{n+1} \le 0$, and the highest weights satisfy:
    \begin{enumerate}
        \item \emph{Case $n=2$:} 
        \[
            2z_1 \ge z_3 - z_4 \ge 2z_2 \ge -z_3 - z_4 \ge 0.
        \]
        
        \item \emph{Case $n \ge 3$:} There exists an integer $k \ge 0$ such that $2k = z_1 + z_2 + z_n + z_{n+2}$. The inner weights are constant, $z_3 = \dots = z_{n-1} = k$ (for $n \ge 4)$, and the boundary weights satisfy:
        \begin{align*}
            \frac{z_{n+1} - z_{n+2}}{2} & \le z_1 \le \frac{z_{n+1} - z_{n+2}}{2} + k, \\
            \frac{z_{n+1} - z_{n+2}}{2} + k - z_1 & \le z_n \le k, \\
            k & \le z_2 \le \frac{z_{n+1} - z_{n+2}}{2} + k - z_n.
        \end{align*}
    \end{enumerate}
    The multiplicity is exactly one.
\end{prop}

\begin{proof}
    We apply Theorem \ref{thm:operational_branching} and lift the representation $\varrho_{\SU(n+1)}(z_1, \dots, z_n)$ to \linebreak$\varrho_{\U(n+1)}(z_1, \dots, z_n, 0)$. The isotropy group of the base is $S(\U(n-1) \times \U(2))$, embedded block-diagonally. For the representation to be spherical, its restriction to this subgroup must contain the representation $\varrho_{\SU(n-1)}(\mathbf{0}) \otimes \varrho_{\U(2)}(\mathbf{y})$ such that the $\U(2)$-factor is dual to the external fiber representation $\varrho_{\SU(2)}(z_{n+1}) \otimes \varrho_{\U(1)}(z_{n+2})$.
    
    The condition for the $\U(n-1)$-factor forces the internal representation $\varrho_{\U(n-1)}(\mathbf{x})$ to be one-dimensional, i.e., $x_1 = \dots = x_{n-1} = k$.
    For the $\U(2)$-factor, duality implies that the restriction to the semisimple part $\SU(2)$ must be isomorphic to $\varrho_{\SU(2)}(z_{n+1})$ (since $\SU(2)$-representations are self-dual), while the restriction to the center $\U(1)$ must be inverse to $\varrho_{\U(1)}(z_{n+2})$. In terms of weights, this yields the system:
    \[
        y_1 - y_2 = z_{n+1} \quad \text{and} \quad y_1 + y_2 = -z_{n+2}.
    \]
    Solving for $y_1, y_2$ gives $y_1 = (z_{n+1} - z_{n+2})/2$ and $y_2 = (-z_{n+1} - z_{n+2})/2$.
    
       By the reciprocity lemma (Lemma \ref{lem:reciprocity}), finding this specific component is equivalent to finding $\varrho_{\U(n+1)}(z_1, \dots, z_n, 0)$ in the tensor product $\varrho_{\U(n+1)}(k^{n-1}) \otimes \varrho_{\U(n+1)}(y_1, y_2)$. 
    
    Because we lifted the initial $\SU(n+1)$-representation to a polynomial representation of $\U(n+1)$ (where the lowest weight is zero), the highest weights of the resulting tensor factors must be non-negative integers. In particular, this requires $y_1 \ge y_2 \ge 0$, which is the precise condition necessary to apply the combinatorial rules for the tensor product decomposition. The requirement that $y_1$ and $y_2$ are integers enforces the parity relation $z_{n+1} \equiv z_{n+2} \pmod 2$. The condition $y_2 \ge 0$ translates to $-z_{n+1} - z_{n+2} \ge 0$, which yields $z_{n+2} \le -z_{n+1}$.
    
    The sum relation follows directly from the action of the center of $\mathfrak{u}(n+1)$. By applying the identity matrix $I_{n+1}$ to any highest weight vector in the tensor product decomposition, it acts as a scalar equal to the sum of the highest weights. Therefore, a representation can only occur in the tensor product if its total weight equals the sum of the total weights of the factors. This immediately yields:
    \[
        \sum_{i=1}^n z_i = (n-1)k + y_1 + y_2.
    \]
    Substituting $y_1 + y_2 = -z_{n+2}$ yields the stated sum relation $\sum_{i=1}^n z_i = (n-1)k - z_{n+2}$, which uniquely determines the integer $k$.
    To evaluate the remaining conditions, we proceed by cases:      
    
    Case $n=3$: We must find $\varrho_{\U(4)}(z_1, z_2, z_3, 0)$ within $\varrho_{\U(4)}(k, k) \otimes \varrho_{\U(4)}(y_1, y_2)$. This matches the exact setting of Proposition \ref{prop: tensor product decomposition} with the partition $\nu = (k^2)$ and $\mu = (y_1, y_2)$. Applying the proposition, we obtain the stated inequalities for the outer weights $z_1, z_2,$ and $z_3$. Note that for $n=3$, there is no internal block $z_3 \dots z_{n-1}$ to constrain.
    
    Case $n \ge 4$: The required tensor product is $\varrho_{\U(n+1)}(z_1, \dots, z_n, 0)$ within $\varrho_{\U(n+1)}(k^{n-1}) \otimes \varrho_{\U(n+1)}(y_1, y_2)$. Applying Proposition \ref{prop: tensor product decomposition} with $\nu = (k^{n-1})$ and $\mu = (y_1, y_2)$ immediately forces the internal weights to be constant: $z_3 = \dots = z_{n-1} = k$. The inequalities for the remaining outer weights $z_1, z_2$, and $z_n$ follow identically from the proposition.
    
    Finally, for any $n \ge 3$, the sum relation derived above requires $\sum_{i=1}^n z_i = (n-1)k - z_{n+2}$. Since the $n-3$ inner weights are all equal to $k$, we have $\sum_{i=1}^n z_i = z_1 + z_2 + z_n + (n-3)k$. Equating these two expressions and solving for $k$ immediately reduces to the required identity $2k = z_1 + z_2 + z_n + z_{n+2}$.
    
       Case $n=2$: By the reciprocity lemma, the multiplicity of $\varrho_{\U(3)}(z_1, z_2, 0)$ in the tensor product $\varrho_{\U(3)}(k) \otimes \varrho_{\U(3)}(y_1, y_2)$ is identical to the multiplicity of $\varrho_{\U(2)}(y_1, y_2) \otimes \varrho_{\U(1)}(k)$ in the restriction $\varrho_{\U(3)}(z_1, z_2, 0) \downarrow \U(2) \times \U(1)$. According to Proposition \ref{prop: branching GoWa}, this component appears if and only if the weights interlace $(z_1 \ge y_1 \ge z_2 \ge y_2 \ge 0)$ and the $\U(1)$-weight satisfies $k = (z_1 + z_2) - (y_1 + y_2)$. Substituting $a = y_1$ and $b = y_2$ into the interlacing inequalities and multiplying by $2$ directly yields $2z_1 \ge 2a \ge 2z_2 \ge 2b \ge 0$.
\end{proof}

Based on the algebraic classification of spherical representations, we derive the explicit spectral formula. We use the constraints on the weights to simplify the expressions significantly.
\begin{thm}\label{thm: spectrum complex stiefel}
    The spectrum of the Laplace-Beltrami operator on the complex Stiefel manifolds $(V_2(\mathbb{C}^{n+1}), g_{t_0, t_1, t_2})$ consists of the eigenvalues $\eta$ with multiplicities $\mult(\eta)$ given as follows:

    \begin{enumerate}
        \item \emph{Case $n=2$:} The eigenvalues depend on the integers $z_1\geq z_2\geq 0,$ $-z_4 \geq z_3 \geq 0$ and are given by
        \begin{align*}
            \eta(z_1,z_2,z_3,z_4) &= \frac{1}{t_0} \left( z_1(z_1 + 2) + z_2^2 - \frac{1}{3}(z_1+z_2)^2 - \frac{z_3(z_3+2)}{2} - \frac{3}{2}z_4^2 \right) \\
            &\quad + \frac{z_3(z_3+2)}{2t_1} + \frac{3}{2t_2}z_4^2,
        \end{align*}
        provided the parameters satisfy
        \[
          z_3 \equiv z_4 \pmod 2,\quad   z_1 \ge \frac{z_3 - z_4}{2} \ge z_2 \ge \frac{-z_3 - z_4}{2} \ge 0.
        \]
        The multiplicity is given by:
        \[
            \mult(\eta) = \frac{1}{2}(z_3 + 1)(z_1 - z_2 + 1)(z_1 + 2)(z_2 + 1).
        \]

       \item \emph{Case $n \ge 3$:} The eigenvalues depend on integers satisfying $z_1\geq  z_2\geq z_n$, $-z_{n+2} \geq z_{n+1} \geq 0$. Moreover, $2k:=z_1 + z_2 + z_n + z_{n+2}$ has to be an even, non-negative integer. The eigenvalues are given by
        \begin{align*}
             \eta(z_1,z_2,z_n,z_{n+1},z_{n+2}) &= \frac{1}{t_0} \Bigg[ \sum_{i \in \{1,2,n\}} z_i(z_i + n + 2 - 2i) + (n-3)k^2 \\
             &\quad\quad - \frac{\left((n-1)k - z_{n+2}\right)^2}{n+1} - \frac{z_{n+1}(z_{n+1}+2)}{2} - \frac{n+1}{2(n-1)}z_{n+2}^2 \Bigg] \\
             &\quad + \frac{z_{n+1}(z_{n+1}+2)}{2t_1} + \frac{n+1}{2(n-1)t_2}z_{n+2}^2,
        \end{align*}
provided the parameters satisfy 
        \begin{align*}
        z_{n+1} &\equiv z_{n+2} \pmod 2\\
            \frac{z_{n+1} - z_{n+2}}{2} & \le z_1 \le \frac{z_{n+1} - z_{n+2}}{2} + k, \\
            \frac{z_{n+1} - z_{n+2}}{2} + k - z_1 & \le z_n \le k, \\
            k & \le z_2 \le \frac{z_{n+1} - z_{n+2}}{2} + k - z_n.
        \end{align*}
        The multiplicity is given by:
        \begin{align*}
            \mult(\eta) &= \frac{1}{n(n-1)^2(n-2)^3} (z_{n+1} + 1)(z_1 + n)(z_2 + n - 1)(z_n + 1) \\[1mm]
            &\quad \cdot (z_1 - z_2 + 1)(z_1 - z_n + n - 1)(z_2 - z_n + n - 2) \\[1mm]
            &\quad \cdot \binom{z_1 - k + n - 2}{n-3} \binom{z_2 - k + n - 3}{n-3} \binom{k - z_n + n - 3}{n-3} \binom{k + n - 2}{n-3}.
        \end{align*}
    \end{enumerate}
\end{thm}

\begin{proof}
    The eigenvalue formula is derived by applying the general result for naturally reductive spaces from \cite{AgricolaHenkel25}. The Casimir eigenvalue of the representation $\varrho_{\SU(n+1)}(z_1, \dots, z_n)$ with respect to the metric $Q$ has been computed in the proof of Theorem \ref{thm: spec Al 3 alpha delta sasaki} and is given by
    \begin{equation}\label{eq Casimir SU}
        \sum_{i=1}^{n} z_i^2 + \sum_{i=1}^{n}(n+2-2i) z_i - \frac{1}{n+1}\left(\sum_{i=1}^{n}z_i\right)^2.
    \end{equation} For the specific highest weight $(z_1, z_2, k, \dots, k, z_n)$ stated in Proposition \ref{prop:spherical complex Stiefel}, the general expression simplifies by separating the variable terms $(i=1, 2, n)$ from the constant block $(3 \le i \le n-1)$. The linear term of the block in Equation \eqref{eq Casimir SU} vanishes because $\sum_{i=3}^{n-1} (n+2-2i) = 0$, and the sum of squares contributes $(n-3)k^2$. Using the sum relation $\sum_{i=1}^n z_i = (n-1)k - z_{n+2}$ from Proposition \ref{prop:spherical complex Stiefel}, the squared sum term in Equation \eqref{eq Casimir SU} becomes $((n-1)k - z_{n+2})^2/(n+1)$.
    
    The Casimir eigenvalue for the $\SU(2)$ fiber representation $\varrho_{\SU(2)}(z_{n+1})$ is $z_{n+1}(z_{n+1}+2)/2$. The central $U(1)$ factor of the vertical space is generated by the element $Z \in \mathfrak{su}(n+1)$, which is the metric dual of the functional $\lambda_{n+2}$ with respect to $Q$, measuring the central weight $z_{n+2}$. It is given by:
\begin{align*}
    Z = \operatorname{diag} \left( \underbrace{-\frac{2}{n+1}, \dots, -\frac{2}{n+1}}_{n-1}, \underbrace{\frac{n-1}{n+1}, \frac{n-1}{n+1}}_{2} \right).
\end{align*}
With respect to the bi-invariant metric $Q(X,Y) = -\tr(XY)$, its squared norm is:
\begin{align*}
    Q(Z, Z) = (n-1) \left( \frac{2}{n+1} \right)^2 + 2 \left( \frac{n-1}{n+1} \right)^2 = \frac{2(n-1)}{n+1}.
\end{align*}
The contribution of the $U(1)$ factor to the Laplace spectrum is determined by the Casimir eigenvalue of the representation with weight $z_{n+2}$, namely $\frac{z_{n+2}^2}{Q(Z,Z)}$. The Casimir eigenvalue for a representation of weight $z_{n+2}$ is therefore $\frac{n+1}{2(n-1)} z_{n+2}^2$.
    
    Subtracting the fiber contributions from the total Casimir eigenvalue and adding them back with the respective scaling factors $1/t_1$ and $1/t_2$ yields the stated formula.
    
   The multiplicity is the product of the dimensions of the $\SU(n+1)$, $\SU(2)$, and $\U(1)$ representations. Since the $\U(1)$ representation is one-dimensional and the $\SU(2)$ dimension is $z_{n+1} + 1$, the total multiplicity reduces to scaling the dimension of the base representation $\varrho_{\SU(n+1)}$ by the factor $(z_{n+1} + 1)$. 
    
    For $n=2$, the base group is $\SU(3)$. Evaluating the Weyl dimension formula for the highest weight $(z_1, z_2)$ directly yields the explicitly stated polynomial. 
    
For $n \ge 3$, the dimension of the $\SU(n+1)$ representation is computed completely analogously to the derivation in Theorem \ref{thm: spec Al 3 alpha delta sasaki}. Because the fiber representation $\varrho_{\U(1)}(z_{n+2})$ is one-dimensional, the total multiplicity depends entirely on $\varrho_{\SU(n+1)}$ and the $\SU(2)$ factor. Since both geometric cases require a constant inner block of $n-3$ weights, the combinatorial structure of the Weyl dimension formula is identical. Consequently, the multiplicity formula from Theorem \ref{thm: spec Al 3 alpha delta sasaki} applies directly by replacing the inner weight parameter $\frac{z_1+z_2+z_n}{4}$ with $k = \frac{z_1 + z_2 + z_n+z_{n+2}}{2}$.
\end{proof}
\subsection{Real Stiefel manifolds $V_2(\mathbb{R}^{m})$}
The real Stiefel manifold $V_2(\mathbb{R}^{m}) \cong \SO(m)/\SO(m-2)$ fibers over the oriented Grassmannian $G_2(\mathbb{R}^m)$ with fiber $H \cong \SO(2)$. Topologically, the real Stiefel manifolds $V_2(\mathbb{R}^m)$ are Spin manifolds. Geometrically, they can be naturally identified with the unit tangent bundles of the spheres $T^1 S^{m-1}$, meaning they carry a standard contact and Sasaki structure \cite{GonzalezDavila22}.

The geometry defined by the canonical variation $g_{t_0, t_1}$ remains an active area of research: For example, on the 7-dimensional Stiefel manifold $V_2(\mathbb{R}^5)$, Moreno and Portilla \cite{MorenoPortilla26} recently demonstrated that the canonical variation naturally carries invariant coclosed $G_2$-structures. In their work, varying the metric parameter along the fiber serves as the primary tool to classify and study the rigidity of homogeneous $G_2$-instantons. 

We consider the family of metrics $g_{t_0, t_1}$ defined by rescaling the bi-invariant background metric $Q(X,Y) = -\frac{1}{2}\tr(XY)$ by $t_0$ on the horizontal subspace $\mathfrak{m}$ and by $t_1$ on the vertical subspace $\mathfrak{h}\cong \so(2)$:
\begin{align}\label{eq:real_stiefel_metric}
    g_{t_0, t_1} = t_0 \cdot Q\myrestriction_{\mathfrak{m}} + t_1 \cdot Q\myrestriction_{\so(2)}.
\end{align}
For $t_0 = t_1$, this recovers the normal homogeneous metric. We denote the ratio of the scaling factors by $t = t_1/t_0$.
We now determine the spherical representations required to compute the spectrum of these metrics. 
\begin{prop}[{\cite{Jensen73, Jensen75},\cite[Thm. 1]{Kerr98}}]\label{prop: Einstein real Stiefel}
    Let $m \ge 3$. There is (up to scaling) exactly one $\SO(m)$-invariant Einstein metric on $V_2(\R^m)$. This metric is contained in the canonical variation $g_{t_0, t_1}$ and is given by the ratio
        \begin{align*}
           t=\frac{t_1}{t_0} = \frac{2(m-2)}{m-1}.
        \end{align*}
        The Einstein constant is given by $$\Lambda=\frac{(m-2)^2}{t_0(m-1)} .$$
\end{prop}
We briefly contextualize this result.
\begin{rem}
This condition is derived from \cite[Prop. 11]{Jensen73} with the parameters $(c=0, r=1, k=\frac{m-2}{2})$, identifying the parameter $t^2$ in \cite{Jensen73} with our ratio $t$.
        \begin{enumerate}
            \item For $m=3$, the solution is $t=1$, which recovers the normal homogeneous metric of constant curvature on $\SO(3)$.
        \item For $m > 3$, the normal homogeneous $(t=1)$ is not Einstein.
    \end{enumerate}
    The specific pinching constants $\delta = K_{\min}/K_{\max}$ can be found in\cite[Prop. 5]{Jensen75}.
\end{rem}
We denote by $n$ the rank of $\SO(m)$ and parameterize the irreducible representations of $\SO(m)$ by their highest weights $\varrho_{\SO(m)}(z_1,\dots,z_n)$ and those of the fiber group by $\varrho_{\SO(2)}(z_{n+1})$, where $z_{n+1} \in \mathbb{Z}$.
To explicitly compute the Laplace-Beltrami spectrum for this entire family of deformed metrics, we must determine the algebraically admissible representations.

\begin{prop}\label{prop: spherical real Stiefel}
    Let $m \ge 3$ and let $n$ denote the rank of $\SO(m)$. The irreducible representation 
    \[
    \varrho_{\SO(m)}(z_1, \dots, z_n) \otimes \varrho_{\SO(2)}(z_{n+1})
    \]
    is $ \SO(m-2)\times \Delta \SO(2)$-spherical if and only if the highest weight satisfies
    \[
    z_3 = \dots = z_n = 0 \quad (\text{for } m \ge 6)
    \]
    and the fiber weight $z_{n+1}$ satisfies the condition:
    \[
    |z_{n+1}| \le z_1 - |z_2| \quad \text{and} \quad z_{n+1} \equiv z_1 - z_2 \pmod 2,
    \]
    under the convention that $z_2 = 0$ for $m=3$. In all valid cases, the multiplicity is exactly one.
\end{prop}
\begin{proof}
 According to Theorem \ref{thm:operational_branching}, the representation is spherical if and only if the restriction of $\varrho_{\SO(m)}(z_1,\dots,z_n)$ to the base isotropy group $K \cdot H \cong \SO(m-2) \times \SO(2)$ contains the component $1_K \otimes \varrho_{\SO(2)}(z_{n+1})^*$. Since the dual of $\varrho_{\SO(2)}(z_{n+1})$ is exactly $\varrho_{\SO(2)}(-z_{n+1})$, we require the restriction to contain $1_{\SO(m-2)} \otimes \varrho_{\SO(2)}(-z_{n+1})$.
    
  For $m=3$, the base group is $\SO(3)$ (rank $n=1)$ and the base isotropy is trivial. The branching from $\SO(3)$ to the fiber group $\SO(2)$ is governed by the classical branching rules (see Proposition \ref{prop: branching GoWa} for $\Spin(3)\downarrow \Spin(2))$. The irreducible representation $\varrho_{\SO(3)}(z_1)$ decomposes into $\SO(2)$-representations of weight $-z_{n+1}$ exactly when $|-z_{n+1}| \le z_1$. Because the representation lifts from $\SO(3)$, $z_1$ and $z_{n+1}$ must have the same parity, yielding $z_{n+1} \equiv z_1 \pmod 2$. Setting $z_2 = 0$, this precisely matches the stated inequalities and parity constraints. Each valid fiber weight appears with multiplicity one.

    For $m \ge 4$, this branching problem is analyzed in \cite{Tsukamoto81}. We evaluate the conditions of \cite[Thm. 1.2]{Tsukamoto81} (for odd $m)$ and \cite[Thm. 1.1]{Tsukamoto81} (for even $m)$ under the constraint that the representation on $\SO(m-2)$ is trivial. In the notation of \cite{Tsukamoto81}, this corresponds to setting the subgroup parameters $k_1 = \dots = 0$, while identifying the parameter $k_0$ with the negative fiber weight $-z_{n+1}$.

    The first necessary condition in \cite{Tsukamoto81} imposes interlacing inequalities on the weights, specifically $h_{i-1} \ge k_i \ge h_{i+1}$ (which translates to $z_i \ge k_i \ge z_{i+2})$. Substituting $k_i=0$ for $i \ge 1$ yields $z_1 \ge 0 \ge z_3$. Since highest weights are dominant, this forces $z_3 = 0$, and inductively $z_j = 0$ for all $n\geq j \ge 3$. Thus, only representations with highest weights of the form $ (z_1, z_2)$ can be spherical.

    The second condition in \cite{Tsukamoto81} determines the multiplicity as the coefficient of $X^{z_{n+1}}$ in a specific generating function. For our case $(z_3=\dots=0$ and $k_1=\dots=0)$, the auxiliary parameters $l_i$ simplify to $l_0 = z_1 - |z_2|$ and $l_i = 0$ for $i \ge 1$. Although the generating functions provided in \cite{Tsukamoto81} differ for even and odd $m$, in our specific case—where the representation on the $\SO(m-2)$-factor is trivial—the additional terms cancel out. Consequently, for both parities of $m$, the generating function reduces to the same finite geometric series:
    \[
   \frac{X^{z_1-z_2+1}-X^{-(z_1-z_2+1)}}{X-X^{-1}} =\sum_{j=0}^{z_1-z_2} X^{(z_1-z_2) - 2j}.
    \]
    The negative fiber weight $-z_{n+1}$ appears in the branching if and only if it corresponds to one of the exponents in this series. This is equivalent to the conditions $|z_{n+1}| \le z_1 - z_2$ and $z_{n+1} \equiv z_1 - z_2 \pmod 2$. If these conditions are satisfied, the coefficient is exactly 1, implying that there is only one spherical vector.
\end{proof}
With the spherical representations classified, we can now apply Theorem \ref{thm:operational_branching} to obtain the explicit eigenvalues of the Laplace-Beltrami operator.

\begin{thm}\label{thm: spectrum real stiefel}
    The spectrum of the Laplace-Beltrami operator on $(V_2(\mathbb{R}^{m}), g_{t_0, t_1})$ for $m \ge 3$ consists of the eigenvalues
    \begin{align*}
        \eta(z_1,z_2, z_{n+1}) = \frac{1}{2t_0} \left( z_1(z_1 + m - 2) + z_2(z_2 + m - 4) -  z_{n+1}^2 \right) + \frac{1}{2t_1}  z_{n+1}^2,
    \end{align*}
 where the parameters $z_1, z_2, z_{n+1} \in \mathbb{Z}$ satisfy $z_1 \ge |z_2|$ (with $z_2 \ge 0$ for $m \ge 5$, and the convention $z_2=0$ if $m=3$) and fulfill the conditions:
    \[
    |z_{n+1}| \le z_1 - |z_2| \quad \text{and} \quad z_{n+1} \equiv z_1 - z_2 \pmod 2.
    \]
    The multiplicity of each eigenvalue is given by:
    \begin{enumerate}
        \item For $m=3$: 
        $$\mult(\eta(z_1))=2z_1 + 1$$
        \item For $m=4$: 
        $$\mult(\eta(z_1,z_2))=(z_1 + z_2 + 1)(z_1 - z_2 + 1) $$
        \item For $m\geq 5:$
        \begin{align*}
       \mult(\eta(z_1,z_2))= &\frac{(z_1+z_2 + m-3)(z_1-z_2 + 1)\cdot(2z_1+m-2)(2z_2+m-4)}{(m-2)(m-3)(m-4)^2} \\ 
        & \cdot
        \binom{z_1+m-4}{m-5}\binom{z_2+m-5}{m-5}.
    \end{align*}
    \end{enumerate}
\end{thm}

\begin{proof}
    By Theorem \ref{thm:operational_branching}, the eigenvalue associated with an admissible representation $\varrho_{\SO(m)} \otimes \varrho_{\SO(2)}$ is obtained by substituting the respective Casimir constants into the general formula for the canonical variation:
    \[
        \eta = \frac{c_Q(\varrho_{\SO(m)}) - c_{Q|_\h}(\varrho_{\SO(2)})}{t_0} + \frac{c_{Q|_\h}(\varrho_{\SO(2)})}{t_1}.
    \]
    As computed in the proof of Theorem \ref{thm: spectrum type B,D}, the Casimir constant for the $\SO(m)$-representation with highest weight $(z_1, z_2, 0, \dots, 0)$ evaluated with respect to the bi-invariant metric $Q(X,Y) = -\frac{1}{2}\tr(XY)$ is given by
    \[
        c_Q(\varrho_{\SO(m)}) = \frac{1}{2} \big( z_1(z_1 + m - 2) + z_2(z_2 + m - 4) \big).
    \]
     The Casimir eigenvalue on the vertical subspace is simply $c_{Q|_{\so(2)}}(\varrho_{\SO(2)}) = \frac{1}{2} z_{n+1}^2$. Substituting these two constants into the eigenvalue equation directly yields the stated formula for $\eta(z_1, z_2, z_{n+1})$.

    To determine the multiplicity of each eigenvalue, we recall that the total multiplicity is the product of the branching multiplicity and the dimensions of the involved representations. By Proposition \ref{prop: spherical real Stiefel}, the branching is multiplicity-free. Furthermore, since $\SO(2)$ is Abelian, the fiber representation $\varrho_{\SO(2)}(z_{n+1})$ is one-dimensional. Consequently, the total multiplicity of the eigenvalue reduces entirely to the dimension of the base representation $\varrho_{\SO(m)}(z_1, z_2)$. These dimensions evaluate exactly to the explicit formulas previously derived in Theorem \ref{thm: spectrum type B,D}.
\end{proof}

\section{Geometric Consequences and Applications}
\label{sec:applications}
In this final section, we apply our spectral formulas to evaluate the linear stability of Perelman's $\nu$-entropy and to establish the algebraic conditions for Yamabe bifurcations.

\subsection{The Scalar Condition for Conformal Stability}
Throughout this subsection, we restrict to closed Einstein manifolds with positive Einstein constant $\Lambda >0$.
When studying the stability of the Einstein metric $g$, one considers infinitesimal metric variations (symmetric 2-tensors). For any compact Einstein manifold other than the round sphere, the classical Berger-Ebin decomposition (cf. \cite[Thm. 2.3]{SchwahnSemmelmann25}) orthogonally splits the tangent space of volume-preserving metrics into three components: pure conformal variations $f \cdot g$, trivial variations arising from diffeomorphisms (Lie derivatives), and transverse traceless $(TT)$ tensors. Because the Einstein-Hilbert action is invariant under diffeomorphisms, stability depends entirely on the conformal and $TT$ components.

Historically, different geometric problems have motivated different notions of stability (see Table \ref{tab:stability_notions} for a summary). Classical $\mathcal{S}$-stability focuses solely on the $TT$-tensors, analyzing the Lichnerowicz Laplacian $\Delta_L$. For the volume-normalized Einstein-Hilbert action however, any Einstein metric is a saddle point: the second variation is strictly positive definite for conformal variations and negative definite on $TT$-tensors (up to a finite coindex). In the Ricci-flow setting, Perelman's shrinker entropy $\nu$
\cite{Perelman02} (see also \cite{CHI04}) provides a more suitable
functional for stability questions, since its second variation is closely
related to dynamical stability and instability under the Ricci flow, see Table \ref{tab:stability_notions}. In this context, a dynamically stable metric acts as an attractor for the flow (see \cite{Kr15, Kr20}).

\begin{table}[ht]
    \centering
    \renewcommand{\arraystretch}{1.5}
    \caption{Overview of Stability Notions for Einstein Metrics 
    (cf. \cite{SchwahnSemmelmann25}). The strict inequalities \(>2\Lambda\) 
    give strict linear stability criteria for the corresponding functional, 
    while a strict inequality below the threshold, either in the \(TT\)-spectrum 
    or in the scalar spectrum, gives the corresponding linear instability. 
    Borderline cases with equality at \(2\Lambda\) require separate analysis.}
    \label{tab:stability_notions}
    \begin{tabularx}{\textwidth}{@{} >{\raggedright\arraybackslash}p{3.5cm} >{\raggedright\arraybackslash}p{3.5cm} X @{}}
        \toprule
        \textbf{Stability Notion} & \textbf{Governing Functional} & \textbf{Required Spectral Conditions} \\
        \midrule
        \textbf{Strict Linear $\mathcal{S}$-Stability} \newline (Einstein-Hilbert) 
        & Total Scalar Curvature $\mathcal{S}$ on $TT$-tensors 
        & $\lambda_L(M, g) > 2\Lambda$ 
        \newline (Lichnerowicz Laplacian on $TT$-tensors) \\
        \hline
        \textbf{Strict Linear $\nu$-Stability} \newline (Perelman \cite{Perelman02}) 
        & Shrinker Entropy $\nu$ on all variations 
        & $\lambda_L(M, g) > 2\Lambda$ \textbf{and} $\eta_1(M,g) > 2\Lambda$ 
        \newline (Requires both tensor and scalar stability) \\
        \hline
        \textbf{Dynamical Stability} 
        & Volume-Normalized Ricci Flow 
        & Strict $\nu$-stability implies dynamical stability. 
        \newline Moreover, strict $\mathcal{S}$-instability or $\eta_1(M,g)<2\Lambda$ 
        implies dynamical instability. \\
        \bottomrule
    \end{tabularx}
\end{table}

The strict linear $\nu$-stability (and thus dynamical stability) of an Einstein metric depends on two independent spectral conditions: the scalar condition $\eta_1(M, g) > 2\Lambda$ and the tensor condition $\lambda_L(M, g) > 2\Lambda$, where $\Lambda$ is the Einstein constant and $\eta_1$ is the first positive eigenvalue of the Laplace-Beltrami operator on functions \cite{CaoHe15}. 

For the context of this paper, we say an Einstein metric is \textit{strictly scalar-stable} if $\eta_1 > 2\Lambda$, \textit{neutrally scalar-stable} if $\eta_1 = 2\Lambda$, and \textit{strictly scalar-unstable} if $\eta_1 < 2\Lambda$. For any closed Einstein manifold not isometric to the round sphere, being strictly scalar-unstable mathematically guarantees that the metric admits destabilizing conformal directions, making it conformally $\nu$-unstable \cite{CHI04, CaoHe15}. By the work of Kr\"oncke \cite{Kr20}, this scalar eigenvalue condition directly determines the dynamical behavior of the metric: if $\eta_1 < 2\Lambda$, the Einstein metric is strictly dynamically unstable under the volume-normalized Ricci flow. Conversely, proving $\eta_1 > 2\Lambda$ satisfies the scalar condition, meaning any potential instability must be of purely tensorial nature. The degenerate situations $\lambda_L(M, g) = 2\Lambda$ or $\eta_1(M,g) = 2\Lambda$ are more complicated. For a more detailed description of the stability of Einstein metrics see \cite{{SchwahnSemmelmann25}}.

While controlling the Lichnerowicz Laplacian on tensors remains difficult, our results determine the scalar condition. Recently, Nagy and Semmelmann \cite{Semmelmann Nagy} provided general lower bounds for the first eigenvalue of the Laplace-Beltrami operator and the associated horizontal Laplacian on canonical variations of 3-Sasaki manifolds. We extract the first eigenvalue $\eta_1$ for all classical series:

\begin{cor}[First Eigenvalues of 3-$(\alpha,\delta)$-Sasaki Manifolds]\label{thm:lambda_1_sasaki}
    For any parameter choice $\alpha, \delta > 0$, the first positive eigenvalue $\eta_1$ of the Laplace-Beltrami operator is exactly given by:
    \begin{enumerate}
        \item \textit{Type A} $(\SU(n+1)/S(\U(n-1)\times \U(1))$, $n \ge 2$, cf. Theorem \ref{thm: spec Al 3 alpha delta sasaki}):
        \[
            \eta_1(\alpha,\delta) = \min \big\{ 8\alpha\delta(n+1), \,\, 8\delta\big(\alpha(n-1) + \delta\big) \big\}.
        \]
        \item \textit{Type B/D} $(\SO(m)/(\SO(m-4)\times \SU(2))$, $m \ge 5$, cf. Theorem \ref{thm: spectrum type B,D}):
        \[
            \eta_1(\alpha,\delta) = \min \big\{ 4\alpha\delta(m-1), \,\, 8\delta\big(\alpha(m-4) + \delta\big) \big\}.
        \]
        \item \textit{Type C: Sphere} $(S^{4n-1} \cong \Sp(n)/\Sp(n-1)$, $n \ge 2)$:
        \[
            \eta_1(\alpha,\delta) = \min \big\{ 8\alpha\delta n, \,\, 4\alpha\delta(n-1) + 3\delta^2 \big\}.
        \]
        \item \textit{Type C: Projective Space} $(\mathbb{R}P^{4n-1} \cong \Sp(n)/(\Sp(n-1)\times \Z_2)$, $n \ge 2)$:
        \[
             \eta_1(\alpha,\delta)  = \min \big\{ 8\alpha\delta n, \,\, 8\delta\big(\alpha(n-1) + \delta\big) \big\}.
        \]
    \end{enumerate}
\end{cor}
\begin{proof}
    For Type A: To minimize the basic eigenvalue $(z_{n+1}=0)$, we choose the minimal coupling $k=1$, forcing $\varrho_G$ to have highest weight $(2,1,\dots,1)$. This yields $\eta_{\mathrm{basic}} = 8\alpha\delta(n+1)$ (this basic eigenvalue explicitly attains the lower bound $\eta_1^B \ge 2\alpha\delta \cdot 4(n+1)$ given in \cite[Cor. 3.10]{Semmelmann Nagy}). Using the same $\varrho_G$ but coupling to the minimal non-trivial fiber representation $\varrho_H(2)$ satisfies all interlacing inequalities and yields $\eta_{\mathrm{mixed}} = 8\alpha\delta(n-1) + 8\delta^2$. Any other valid representation yields strictly larger values.
    
    For Type B/D: For the basic eigenvalue, the minimum occurs at $\varrho_G(1,0,\dots,0)$, yielding $\eta_{\mathrm{basic}} = 4\alpha\delta(m-1)$. For the mixed eigenvalue $(\varrho_H(2))$, the constraint $z_2 \ge z_{n+1}/2$ forces $z_2 \ge 1$. The minimum occurs at $\varrho_G(1,1,0,\dots,0)$, yielding $\eta_{\mathrm{mixed}} = 4\alpha\delta(2m-4) + 8\delta(\delta-2\alpha) = 8\delta(\alpha(m-4) + \delta)$.
    
   For Type C (sphere): The basic eigenvalue is minimized by $\varrho_G(1,1,0,\dots,0)$, yielding $\eta_{\mathrm{basic}} = 8\alpha\delta n$. The mixed eigenvalue is minimized at $\varrho_G(1,0,\dots,0)$ (coupling to $\varrho_H(1))$, yielding $\eta_{\mathrm{mixed}} = 4\alpha\delta(n-1) + 3\delta^2$.
    
    For Type C (projective space): The basic representation $\varrho_G(1,1,0,\dots,0)$ satisfies the parity constraint $z_1 \equiv z_2 \pmod 2$ and remains the minimum, yielding $\eta_{\mathrm{basic}} = 8\alpha\delta n$. However, the minimal mixed representation $\varrho_G(1,0,\dots,0)$ is forbidden by the parity constraint. The next minimal admissible mixed representation is $\varrho_G(2,0,\dots,0)$ (coupling to $\varrho_H(2))$, which evaluates to $\eta_{\mathrm{mixed}} = 8\delta(\alpha(n-1) + \delta)$.
\end{proof}

Using Theorem \ref{thm:lambda_1_sasaki}, we evaluate the scalar stability for the two unique Einstein metrics (cf. Proposition \ref{prop:sasaki_einstein}) across all families.

\begin{cor}[Scalar Stability of 3-Sasaki Einstein Metrics]
    For the Einstein metrics on the 3-$(\alpha,\delta)$-Sasaki families, the following holds:
    \begin{enumerate}
        \item \textit{Type A:} Both the 3-$\alpha$-Sasaki metric $(\delta=\alpha)$ and the squashed Einstein metric $(\delta=(2n+1)\alpha)$ are strictly scalar-stable $(\eta_1 > 2\Lambda)$ for all $n \ge 2$.
        \item \textit{Type B/D:} The 3-$\alpha$-Sasaki metric is strictly scalar-stable for $m=5$, neutrally scalar-stable $(\eta_1 = 2\Lambda)$ for $m=6$, and strictly scalar-unstable $(\eta_1 < 2\Lambda)$ for $m \ge 7$. The squashed Einstein metric $(\delta=(2m-5)\alpha)$ is strictly scalar-unstable, making it conformally $\nu$-unstable for all $m \ge 5$.
        \item \textit{Type C:} The $3$-$\alpha$-Sasaki metric $(\delta=\alpha)$ on the sphere $S^{4n-1}$ is the round sphere (up to scaling) and therefore satisfies the classical inequality $\eta_1 < 2\Lambda$ (cf. Remark \ref{rem:sphere}). Descending to the projective space $\mathbb{R}P^{4n-1}$ geometrically eliminates these eigenfunctions (they are odd under the antipodal map), so that the $3$-$\alpha$-Sasaki metric on $\mathbb{R}P^{4n-1}$ is strictly scalar-stable $(\eta_1 > 2\Lambda)$. The squashed Einstein metric $(\delta=(2n+1)\alpha)$ remains strictly scalar-unstable for both topologies.
    \end{enumerate}
\end{cor}
\begin{proof}
    We explicitly compare $\eta_1$ against $2\Lambda$. For Type A squashed, $\delta > 2\alpha$, so the basic eigenvalue dominates, yielding $\eta_1 = 8\alpha^2(2n^2+3n+1)$. The threshold is $2\Lambda_2 = 4\alpha^2(4n^2+6n-1)$, yielding $\eta_1 - 2\Lambda_2 = 12\alpha^2 > 0$. 
    For Type B/D squashed, $\delta \ge 5\alpha > 2\alpha$, giving $\eta_1 = 4\alpha^2(2m^2-7m+5)$. The threshold is $2\Lambda_2 = 4\alpha^2(4m^2-18m+17)$, yielding $\eta_1 - 2\Lambda_2 = 4\alpha^2(-2m^2+11m-12) < 0$ for all $m \ge 5$. 
    For Type C 3-$\alpha$-Sasaki (round sphere), $\eta_1 = (4n-1)\alpha^2 < (8n-4)\alpha^2 = 2\Lambda$. 
    
     In contrast, for $\mathbb{R}P^{4n-1}$, the first eigenvalue jumps to $\eta_1 = 8n\alpha^2$, which strictly dominates $2\Lambda = (8n-4)\alpha^2$, proving strict scalar stability. 
    For the squashed metric on either topology, the basic eigenvalue dominates: $\eta_1 = 8n(2n+1)\alpha^2$. The threshold is $2\Lambda_2 = 4(4n^2+6n-1)\alpha^2$. The difference is $\alpha^2(-16n+4) < 0$, yielding genuine $\nu$-instability for both spaces.
\end{proof}
\begin{rem}\label{rem:sphere}    We briefly clarify why the round sphere $(S^n, g_0)$ represents a unique geometric exception in this stability analysis. In general, an eigenvalue $\eta_1 < 2\Lambda$ implies that a conformal variation strictly increases the entropy, causing $\nu$-instability. For the round sphere, the first eigenvalue evaluates to $\eta_1 = \frac{n}{n-1}\Lambda < 2\Lambda$. 
    
    However, the first eigenfunctions on $S^n$ naturally generate non-isometric conformal transformations. Consequently, the space of pure conformal variations and the space of trivial variations (diffeomorphisms) intersect non-trivially \cite[Rem. 2.4]{SchwahnSemmelmann25}. Since the $\nu$-entropy is diffeomorphism-invariant, these specific conformal directions are neutral rather than destabilizing. Thus, the round sphere is in fact geometrically $\nu$-stable \cite[Ex. 2.2]{CHI04}.
    
    This phenomenon is strictly limited to the sphere. By the resolution of the Lichnerowicz conjecture \cite{Ferrand96, Obata62}, the conformal transformation group of any compact Riemannian manifold not conformally equivalent to $S^n$ is inessential, meaning it reduces to the isometry group. Therefore, for the geometric families studied in this paper, the relevant variation spaces intersect trivially. Thus, scalar instability $(\eta_1 < 2\Lambda)$ guarantees genuine $\nu$-instability.
\end{rem}

Next, we state the first eigenvalue for the Stiefel manifolds.

\begin{cor}[First Eigenvalues of Stiefel Manifolds]\label{cor:lambda_1_stiefel}
    The first positive eigenvalue of the Laplace-Beltrami operator is exactly given by:
    \begin{enumerate}
       \item \textit{Real Stiefel} $(V_2(\mathbb{R}^m)$, canonical variation $g_{t_0, t_1})$:
        \begin{align*}
            \eta_1(t_0, t_1) &= \min \left\{ \frac{3}{t_0}, \,\, \frac{1}{2t_0} + \frac{1}{2t_1} \right\} \quad \text{for } m=3, \\
            \eta_1(t_0, t_1) &= \min \left\{ \frac{m-2}{t_0}, \,\, \frac{m-2}{2t_0} + \frac{1}{2t_1} \right\} \quad \text{for } m \ge 4.
        \end{align*}
        \item \textit{Complex Stiefel} $(V_2(\mathbb{C}^{n+1})$ for $n \ge 2$, Jensen type metric $g_{t_0, t_1, t_2})$:
        \begin{align*}
           &\eta_1(t_0, t_1, t_2) \\
            &= \min \left\{ \frac{2(n-1)(n+2)}{t_0(n+1)}, \,\, \frac{1}{t_0} \left( \frac{n^2+2n}{n+1} - \frac{3}{2} - \frac{n+1}{2(n-1)} \right) + \frac{3}{2t_1} + \frac{n+1}{2(n-1)t_2} \right\}. 
        \end{align*}
    \end{enumerate}
\end{cor}
\begin{proof}
The spectrum of the real Stiefel manifolds has been computed in Theorem \ref{thm: spectrum real stiefel}. The basic representations $(\varrho_{\SO(2)}(0))$ require the parity condition $z_1 \equiv z_2 \pmod 2$. For $m \ge 4$, the minimum occurs at the highest weight $\varrho_{\SO(m)}(1,1,0\dots,0)$, yielding $\eta_{\mathrm{basic}} = \frac{m-2}{t_0}$. For $m=3$ (where $z_2$ does not exist), the minimum occurs at the even weight $z_1=2$, yielding $\eta_{\mathrm{basic}} = \frac{3}{t_0}$. The mixed representations require $z_{n+1} \equiv z_1-z_2 \pmod 2$ and $|z_{n+1}| \le z_1-z_2$. For all $m \ge 3$, the minimum for the lowest non-trivial fiber weight $z_{n+1}=1$ occurs at $\varrho_{\SO(m)}(1,0,\dots,0)$, yielding $\eta_{\mathrm{mixed}} = \frac{m-2}{2t_0} + \frac{1}{2t_1}$.
    
    The spectrum of the complex Stiefel manifolds has been computed in Theorem \ref{thm: spectrum complex stiefel}: The minimal basic representation $(\varrho_H(0,0))$ forces $k=1$ and highest weight\linebreak $\varrho_{\SU(n+1)}(1,1,\dots,1,0)$, yielding $\eta_{\mathrm{basic}}$. The minimal mixed representation requires fiber weights $z_{n+1}=1, z_{n+2}=-1$, forcing $k=0$ and highest weight $\varrho_{\SU(n+1)}(1,0,\dots,0)$, yielding $\eta_{\mathrm{mixed}}$.
\end{proof}

For the specific case of the Einstein metric on $V_2(\R^m)$, our exact formula confirms the conformal $\nu$-instability recently obtained by Wang and Wang \cite[Ex. 4.7]{WangWang22} via basic eigenfunctions.

\begin{cor}
    The unique $\SO(m)$-invariant Einstein metric on $V_2(\mathbb{R}^m)$ $(m \ge 4)$ is strictly scalar-unstable $(\eta_1 < 2\Lambda)$ and therefore dynamically unstable under the Ricci flow. For $m=3$, the unique Einstein metric is the standard normal metric on $V_2(\mathbb{R}^3) \cong \mathbb{R}P^3$, which is neutrally stable $(\eta_1 = 2\Lambda)$.
    
    Both $\SU(n+1)$-invariant Einstein metrics on $V_2(\mathbb{C}^{n+1})$ $(n \ge 2)$ are strictly scalar-unstable $(\eta_1 < 2\Lambda)$.
\end{cor}
\begin{proof}
For $m=3$, the real Stiefel manifold is isometric to the real projective space $\mathbb{R}P^3$. The Einstein condition enforces $t_1 = t_0$. Evaluating Corollary \ref{cor:lambda_1_stiefel} for $m=3$ yields the first positive eigenvalue $\eta_1 = \eta_{\mathrm{mixed}} = \frac{1}{2t_0} + \frac{1}{2t_0} = \frac{1}{t_0}$, which coincides exactly with the threshold $2\Lambda = \frac{1}{t_0}$, proving neutral stability.

    For $m\geq 4$, the Einstein metric is given by $t_1 = t_0 \frac{2(m-2)}{m-1}$ and $2\Lambda = \frac{2(m-2)^2}{t_0(m-1)}$. The mixed eigenvalue constitutes the minimum: $\eta_1 = \eta_{\mathrm{mixed}} = \frac{2(m-2)^2 + m - 1}{4t_0(m-2)}$. The instability condition $\eta_1 < 2\Lambda$ translates algebraically to $\frac{2(m-2)^2 + m - 1}{4(m-2)} < \frac{2(m-2)^2}{m-1}$. Substituting $x = m-2 \ge 2$, this inequality simplifies to $3x^2 + 2x + 1 < 6x^3$, which holds strictly for all $x \ge 2$.
    
   For the Complex Stiefel Einstein metrics, the threshold for scalar stability is $2\Lambda = \frac{(n+1)t_2}{t_0^2}$. We analyze the two metrics defined by the sign choices in Proposition \ref{prop: Einstein complex Stiefel}. Let $x = \sqrt{\frac{n^2-2n}{n^2-1}} \in (0,1)$.

    For the second Einstein metric (corresponding to the lower signs: $t_1 = t_0(1-x)$ and $t_2 = t_0\frac{n-1}{(n+1)(2n-1)}(4n+1+3x))$, we evaluate the difference between the basic eigenvalue and the threshold:
    \begin{align*}
        \eta_{\mathrm{basic}} - 2\Lambda &= \frac{2(n-1)(n+2)}{t_0(n+1)} - \frac{n-1}{t_0(2n-1)}(4n+1+3x) \\
        &= \frac{n-1}{t_0(n+1)(2n-1)} \big[ 2(n+2)(2n-1) - (n+1)(4n+1+3x) \big] \\
        &= \frac{n-1}{t_0(n+1)(2n-1)} \big[ n - 5 - 3x(n+1) \big].
    \end{align*}
    For $n=2$, the parameter $x$ evaluates to $0$, causing the two Einstein metrics to coincide. The term in brackets evaluates directly to $2 - 5 - 0 = -3 < 0$.
    For $n=3$, the term in brackets evaluates to $-2 - 12x < 0$. For all $n \ge 4$, we have $x \ge \sqrt{8/15} \approx 0.73$, so $3x(n+1) > 2.19n + 2.19 > n - 5$, making the bracket strictly negative. Thus, $\eta_{\mathrm{basic}} < 2\Lambda$ unconditionally, proving this metric is scalar-unstable.

    For the first Einstein metric (corresponding to the upper signs: $t_1 = t_0(1+x)$ and $t_2 = t_0\frac{n-1}{(n+1)(2n-1)}(4n+1-3x))$, we evaluate the mixed eigenvalue $\eta_{\mathrm{mixed}}$. By expanding the coefficients and factoring out $1/t_0$, we have:
    \[
        \eta_{\mathrm{mixed}} = \frac{1}{t_0} \left( n - 1 - \frac{2n}{n^2-1} + \frac{3}{2(1+x)} + \frac{(n+1)^2(2n-1)}{2(n-1)^2(4n+1-3x)} \right).
    \]
    To prove $\eta_{\mathrm{mixed}} < 2\Lambda$, we establish a strict algebraic lower bound for the threshold $2\Lambda$ and a strict upper bound for $\eta_{\mathrm{mixed}}$. First, note that for $n \ge 3$, the square root satisfies $0 < x < 1$ because $0 < n^2-2n < n^2-1$.
    
    For the Einstein threshold, the condition $x < 1$ implies $-3x > -3$, which allows us to bound $2\Lambda$ strictly from below:
    \[
        2\Lambda = \frac{n-1}{t_0(2n-1)}(4n+1-3x) > \frac{n-1}{t_0(2n-1)}(4n-2) = \frac{2(n-1)}{t_0} = \frac{2n-2}{t_0}.
    \]
    Conversely, we bound $\eta_{\mathrm{mixed}}$ from above. Since $x > 0$, the first fraction satisfies $\frac{3}{2(1+x)} < \frac{3}{2}$. Furthermore, since $x < 1$, the denominator of the second fraction satisfies $4n+1-3x > 4n-2 = 2(2n-1)$, which yields $\frac{2n-1}{4n+1-3x} < \frac{1}{2}$. Applying these bounds, we obtain:
    \[
        \eta_{\mathrm{mixed}} < \frac{1}{t_0} \left( n - 1 - \frac{2n}{n^2-1} + \frac{3}{2} + \frac{(n+1)^2}{4(n-1)^2} \right).
    \]
    To show that $\eta_{\mathrm{mixed}} < 2\Lambda$, it now suffices to show that this upper bound is strictly less than or equal to the lower bound of $2\Lambda$, which translates to:
    \[
        n - 1 - \frac{2n}{n^2-1} + \frac{3}{2} + \frac{(n+1)^2}{4(n-1)^2} \le 2n - 2 \quad \Longleftrightarrow \quad \frac{3}{2} + \frac{(n+1)^2}{4(n-1)^2} - \frac{2n}{n^2-1} \le n - 1.
    \]
    For $n \ge 3$, the term $\frac{n+1}{n-1} = 1 + \frac{2}{n-1}$ is strictly decreasing and bounded above by $2$ (its value at $n=3)$. Therefore, $\frac{(n+1)^2}{4(n-1)^2} \le \frac{4}{4} = 1$. Since the subtractive term $-\frac{2n}{n^2-1}$ is strictly negative, the entire left-hand side is strictly less than $\frac{3}{2} + 1 = 2.5$. 
    
    For $n \ge 4$, the right-hand side is $n-1 \ge 3 > 2.5$, meaning the inequality holds unconditionally. For the remaining case $n=3$, direct evaluation of the left-hand side yields $\frac{3}{2} + \frac{16}{16} - \frac{6}{8} = 1.75$, which is strictly less than $n-1 = 2$. 

    Thus, $\eta_{\mathrm{mixed}} < 2\Lambda$ holds strictly for all $n \ge 3$. Therefore, both Einstein metrics on the complex Stiefel manifolds are unconditionally strictly scalar-unstable.

\end{proof}

\subsection{Algebraic Mapping of Yamabe Bifurcations and Local Rigidity}

The Yamabe problem seeks to find a metric of constant scalar curvature within a given conformal class. While the existence of such a metric is always guaranteed, uniqueness generally fails. For the homogeneous principal bundles considered here, the canonical variation $g_{\mathbf{t}}$ naturally provides a 1-parameter family of trivial solutions.

 A natural question in geometric analysis is whether a given homogeneous solution is isolated in its conformal class (\textit{locally rigid}) or whether a new branch of non-homogeneous Yamabe solutions bifurcates. The second variation of the restricted Hilbert-Einstein functional is governed by the operator $$\Delta - \frac{\scal}{D-1},$$ where $D$ is the dimension of the manifold (cf. \cite[Sec. 2]{BettiolPiccione13}). Consequently, a metric is a non-degenerate critical point—and thus unconditionally locally rigid—if the threshold $\frac{\scal}{D-1}$ is not an eigenvalue of the Laplace-Beltrami operator. In particular, any metric whose first positive eigenvalue satisfies $\eta_1 > \frac{\scal}{D-1}$ is a strict local minimum and locally rigid.
 
To study non-uniqueness along a parameter family, one considers the \emph{Morse index} of the functional, defined as the number of positive Laplace-Beltrami eigenvalues (counted with multiplicity) strictly less than $\frac{\scal}{D-1}$. As demonstrated by Bettiol and Piccione \cite{BettiolPiccione13} and Otoba and Petean \cite{OtobaPetean20}, a symmetry-breaking bifurcation is guaranteed to occur whenever this Morse index changes. This happens precisely at \emph{degeneracy values}—parameters where an eigenvalue exactly crosses the threshold $\frac{\scal}{D-1}$. At these points, the jump in the Morse index is exactly the multiplicity of the crossing eigenvalue. Mapping out all bifurcation branches therefore requires exact knowledge of the complete spectrum and its multiplicities.

According to a classical theorem of Obata \cite{Obata71}, any Einstein metric not isometric to the round sphere is the unique constant scalar curvature metric in its conformal class, and is thus unconditionally locally rigid. First, we algebraically recover this local rigidity for the 3-$\alpha$-Sasaki metrics from our spectral data.

\begin{cor}
    Let $g$ be the 3-$\alpha$-Sasaki Einstein metric $(\delta=\alpha)$ on a manifold of dimension $D=4d+3$. The stability threshold for local rigidity is exactly $\frac{\scal}{D-1} = \alpha^2(4d+3)$. Consequently:
    \begin{enumerate}
        \item The Type A metric $(d=n-1)$ is locally rigid $(\eta_1 = 8n\alpha^2 > \alpha^2(4n-1))$ for all $n \ge 2$.
        \item The Type B/D metric $(d=m-4)$ is locally rigid $(\eta_1 > \alpha^2(4m-13))$ for all $m \ge 5$.
\item \textit{Type C:} The 3-$\alpha$-Sasaki metric on the sphere $S^{4n-1}$ exactly hits the threshold, $\eta_1 = \alpha^2(4n-1) = \frac{\scal}{D-1}$. The metric is therefore a degenerate critical point and not locally rigid. In contrast, on the projective space $\mathbb{R}P^{4n-1}$, the first eigenvalue is $\eta_1 = 8n\alpha^2 > \frac{\scal}{D-1}$, making the metric strictly locally rigid.
    \end{enumerate}
\end{cor}
\begin{rem}\label{rem:yamabe_sphere_exception}
    The lack of local rigidity on $S^{4n-1}$ arises because the round sphere admits non-isometric conformal transformations, which generate a continuous family of trivial Yamabe solutions (cf. \cite{BettiolPiccione13}). Passing to the quotient $\mathbb{R}P^{4n-1}$ algebraically eliminates these transformations, pushing the first eigenvalue beyond the critical threshold.
\end{rem}

Furthermore, to locate the continuous points of bifurcation along the parameter family, we substitute our eigenvalue formula \eqref{eq:eigenvalue_formula} into the bifurcation condition. For the classical series of 3-$(\alpha,\delta)$-Sasaki manifolds, this yields:
\[
    \delta^2 c_{Q\restriction\su(2)}(\varrho_H) + 2\alpha\delta \big(c_Q(\varrho_G) - c_{Q\restriction\su(2)}(\varrho_H)\big) = \frac{8d(d+2)\alpha\delta - 6d\alpha^2 + 3\delta^2}{2d+1}.
\]
Because Yamabe metrics are invariant under global volume scaling, the bifurcation branches depend solely on the parameter ratio $x = \delta/\alpha > 0$. Dividing the equation by $\alpha^2$, we reduce the entire geometric bifurcation problem to a single quadratic equation:
\begin{equation}\label{eq:bifurcation_master}
    \left( c_{Q\restriction\su(2)}(\varrho_H) - \frac{3}{2d+1} \right) x^2 + 2 \left( c_Q(\varrho_G) - c_{Q\restriction\su(2)}(\varrho_H) - \frac{4d(d+2)}{2d+1} \right) x + \frac{6d}{2d+1} = 0.
\end{equation}

Locating Yamabe bifurcations—especially those that break the structural fiber symmetries $(c_{Q\restriction\su(2)}(\varrho_H) > 0)$—is equivalent to finding positive roots $x$ of \eqref{eq:bifurcation_master}. 

The values of the geometric dimension parameter $d$, as well as the explicit Casimir constants $c_Q$ and $c_{Q\restriction\su(2)}$ evaluated in terms of the highest weights, were established in the proofs of Theorems \ref{thm: spec Al 3 alpha delta sasaki}, \ref{thm: spectrum type B,D}, and \ref{thm: spectrum 3 sasaki type C}. By substituting the algebraically admissible highest weights into \eqref{eq:bifurcation_master}, one obtains the exact discrete parameter ratios $\delta/\alpha$ (the degeneracy values) at which bifurcations occur. As the parameter $x$ crosses such a root, the resulting jump in the Morse index is exactly given by the sum of the multiplicity formulas $\mult(\eta)$ evaluated for the crossing representations. We note that since the scalar curvature blows up as the fibers shrink $(x \to \infty)$, the critical threshold crosses the infinitely many constant basic eigenvalues $(c_{Q\restriction\su(2)}(\varrho_H) = 0)$. Hence, Equation \eqref{eq:bifurcation_master} yields an infinite sequence of symmetry-breaking bifurcations accumulating as $x\rightarrow\infty$.

This algebraic approach to Yamabe bifurcations applies equally well to the families of real and complex Stiefel manifolds presented in Section \ref{sec:Stiefel Manifolds}, requiring only the substitution of their respective scalar curvature functions.

\end{document}